\documentclass[a4paper, 10pt ]{article}
\usepackage[a4paper,left=1.25cm,right=1.25cm,top=2.5cm,bottom=2.5cm]{geometry}
\usepackage{graphicx}
\usepackage{algorithm}
\usepackage{algpseudocode}[compatible]
\usepackage{authblk}
\usepackage{xcolor}    
\usepackage{comment}
\usepackage{indentfirst,csquotes}
\usepackage{ulem}
\usepackage{booktabs}
\usepackage{array}
\usepackage{longtable}
\usepackage{amssymb,amsthm,amsmath}
\usepackage{xcolor,paralist,hyperref,titlesec,fancyhdr,etoolbox}
\usepackage{mathtools}

\newtheorem{theorem}{Theorem}[]
\newtheorem{definition}{Definition}

\newtheorem{assumption}{Assumption}

\newtheorem{proposition}{Proposition}

\usepackage{lipsum}

\begin{document}
\title{A Quadratic Order Reduction - Gaussian Process Ordinary Differential Equation framework for the inference of Large Continuous Dynamical Systems} %%%%%%%%%%%%
\author[1]{Guglielmo Padula\footnote{gpadula@sissa.it}}
\author[2]{Michele Girfoglio\footnote{michele.girfoglio@unipa.it}}
\author[1]{Gianluigi Rozza\footnote{grozza@sissa.it}}
\affil[1]{mathLab, Mathematics Area, SISSA, via Bonomea 265, I-34136 Trieste,
Italy}
\affil[2]{Department of Engineering, University of Palermo, viale delle Scienze 7, I- 90128 Palermo, Italy}

\date{\today}
\maketitle

\begin{abstract}
Forecasting the evolution of complex dynamical systems remains a fundamentally challenging task, primarily due to pronounced nonlinear interactions, high-dimensional state spaces, and the concomitant requirement for rigorous and reliable uncertainty quantification. Contemporary reduced-order modelling (ROM) frameworks frequently exhibit inherent trade-offs among predictive accuracy, numerical stability, and interpretability, and thus often fail to achieve an optimal balance among these competing objectives. To address these limitations, we propose a framework for forecasting complex dynamical systems via a kernel autonomous ordinary differential equation approach based on Gaussian Processes and Quadratic Order Model Reduction. Our base method, the Gaussian Process Ordinary Differential Equations model, allows accurate short-term forecasting with uncertainty quantification, and it provably converges to the real autonomous equation in the smooth case. We integrate it with quadratic order reduced-order modelling and sphere projection for learning the latent dynamics efficiently while preserving stability. 
Numerical experiments demonstrate that our full model outperforms ROM forecasting methods such as Extended Dynamic Mode Decomposition, Bagging Optimised Dynamic Mode Decomposition and Linear and Nonlinear Disambiguation Optimisation in terms of accuracy or computational costs. These results demonstrate the potential of the framework as a robust and stable tool for forecasting complex dynamical systems with rigorous uncertainty quantification.
\end{abstract}
\clearpage
\section{Introduction}
The goal of this paper is to provide a comprehensive and rigorous framework in theory and computation for the development of Digital Shadows \cite{kritzinger_digital_2018} that will allow for efficient and reliable simulations of bounded large-scale dynamical systems. These systems often arise, for example, as solutions to partial differential equations (PDEs) or data-driven representations of complex physical systems such as ocean currents in the Mediterranean Sea or structural loading in engineering infrastructures. Accurate and high-fidelity models of such systems are necessary for preventive decision-making, as they allow for early detection and diagnosis of anomalies and design of proactive interventions. In this framework, Digital Shadows are defined as continuously updated data-driven models that combine sensor data and computational models to provide near real-time forecasts. This capability is particularly critical in applications such as numerical weather prediction, where the provision of accurate and timely forecasts is essential for effective risk assessment and management \cite{brotzge_challenges_2023, merz_impact_2020}. A key challenge in the construction of Digital Shadows lies in the modelling of the underlying dynamics, as it is necessary to balance computational tractability with sufficient expressive power to represent complex, nonlinear interactions in high-dimensional state spaces. Achieving this balance remains a central open problem in scientific computing and engineering across a broad range of engineering applications. \\

Reduced-order Modelling (ROM) techniques constitute a widely adopted class of methods for addressing this problem \cite{tomada_sparse_2025,codega_machine_2025,chen_time_2025,bevanda_koopman-equivariant_2025, gelsomino_comparison_2011}. Approaches based on Proper Orthogonal Decomposition (POD) combined with Galerkin projection \cite{gelsomino_comparison_2011} have been successfully employed for short-term forecasting of large-scale dynamical systems, such as fluid flows and atmospheric processes, frequently in conjunction with machine learning methodologies \cite{khamlich_optimal_2025,rojas_reduced-order_2021,xie_non-intrusive_2019,farenga_latent_2025,daniel_model_2020,san_artificial_2019}. Within this class, Dynamic Mode Decomposition (DMD) and its variants provide notable computational benefits by approximating the system dynamics on a low-dimensional linear subspace \cite{williams_datadriven_2015, sashidhar_bagging_2022,schmid_dynamic_2010,  andreuzzi_dynamic_2023}, with established theoretical convergence guarantees \cite{korda_convergence_2018}. Nonetheless, ROMs can be adversely affected by stability deterioration and the accumulation of numerical errors, particularly when applied to nonlinear dynamical systems or in settings characterised by pronounced parametric variability  \cite{brunton_discovering_2016, grimberg_stability_2020}. These limitations have motivated the development of alternative methodologies, such as Sparse Identification of Nonlinear Dynamics (SINDy) \cite{brunton_discovering_2016} and physics-informed modelling frameworks \cite{baddoo_physics-informed_2023,kaptanoglu_promoting_2021}, which aim to alleviate these issues by explicitly incorporating nonlinear effects and/or relevant physical prior knowledge.\\

To better capture nonlinear dynamics, kernel-based methods have emerged as a promising alternative \cite{baddoo_kernel_2022,pia_surrogate_2025, ensinger_exact_2024, heinonen_learning_2018}. These approaches can implicitly represent complex functions through kernel expansions, enabling flexible modelling in high-dimensional settings. In the context of dynamical systems, kernel-based techniques such as the Linear and Nonlinear Disambiguation Optimisation (LANDO) algorithm have demonstrated the ability to recover both linear and nonlinear components of the dynamics, thus improving robustness while maintaining interpretability \cite{baddoo_kernel_2022}. More broadly, kernel methods enjoy strong theoretical properties: under suitable conditions, they possess universal approximation capabilities, both in the framework of Kernel Ridge Regression \cite{maddalena_deterministic_2021} and Gaussian Processes \cite{lederer_uniform_2021}, with successful applications in forecasting and control \cite{ensinger_exact_2024, maddalena_deterministic_2021, lederer_uniform_2021, capone_gaussian_2022, steinwart_consistency_2009, tuo_improved_2020, wendland_scattered_2004}. These properties make them particularly well-suited for modelling nonlinear dynamical systems in a data-driven setting.\\

Building on these ideas, we introduce the proposed Gaussian Process Ordinary Differential Equation (GPRODE) framework. The central idea is to approximate the drift term of an autonomous ordinary differential equation (ODE) using a Gaussian Process (GP) \cite{rasmussen_gaussian_2005}. Under suitable conditions, the GP posterior mean converges to the true drift function while the posterior variance vanishes as the amount of data increases. As a consequence, the stochastic model converges to the underlying deterministic ODE in the infinite-data limit. Notably, the proposed approach performs inference directly on the drift term, without requiring the application of a numerical integration scheme during training, enabling the decoupling between the numerical schemes used in training and testing. \\

The proposed methodology is related to several existing approaches. In particular, \cite{ensinger_exact_2024} also considers GP-based modelling of dynamical systems, but relies on numerical discretisation of the ODE, whereas our approach operates directly on the continuous-time drift. Similarly, \cite{heinonen_learning_2018} employs Gaussian Processes to learn the drift function using inducing-point approximations and variational inference, but does not incorporate spatial reduction. The strategy in \cite{pia_surrogate_2025} combines Gaussian Processes with latent-space modelling, but uses a linear reduction, whereas we adopt a quadratic reduced-order model. Moreover, in our setting, the derivatives must be estimated from data, while in their framework they are directly available from the reduced model.\\

Despite their flexibility, kernel-based autoregressive models may exhibit instability and often require ad hoc regularisation strategies \cite{xiao_learning_2020}. In the GPRODE setting, this instability arises from the unconstrained nature of the Gaussian Process posterior, which provides no guarantee that the learnt drift will generate bounded trajectories over long-time horizons. To address this issue, we propose an extension of the GPRODE framework. Specifically, we first construct a quadratic reduced-order model, in the spirit of \cite{geelen_operator_2023}, then project the resulting dynamics onto the unit sphere to enforce boundness, and finally apply the GPRODE formulation in the latent space. We refer to this approach as Quadratically Reduced Gaussian Process Ordinary Differential Equations (QGPRODE). We evaluate the proposed methods against several established baselines, including Extended Dynamic Mode Decomposition (EDMD) \cite{williams_datadriven_2015}, Bagging Optimised Dynamic Mode Decomposition (BOPDMD) \cite{sashidhar_bagging_2022}, and LANDO, showing improved accuracy or computational time.\\

The remainder of the paper is organised as follows. Section 2 introduces the GPRODE and QGPRODE methodologies, while Section 3 presents numerical experiments, and Section 4 concludes the paper.

Given the large number of variables employed throughout this paper, we provide here a notation table.

\begin{longtable}{@{} >{\(}l<{\)} p{9cm} @{}}
\caption{Summary of notation.}
\label{tab:notation} \\
\toprule
\multicolumn{1}{l}{\textbf{Symbol}} & \textbf{Description} \\
\midrule
\endfirsthead
\multicolumn{2}{c}{\tablename~\ref{tab:notation} (continued)} \\
\toprule
\multicolumn{1}{l}{\textbf{Symbol}} & \textbf{Description} \\
\midrule
\endhead
\bottomrule
\endfoot

\multicolumn{2}{@{}l}{\textit{Data and state}} \\[2pt]
X(t) \in \mathbb{R}^{m}      & State vector of the dynamical system at time $t$ \\
X_{\mathrm{train}}            & Training dataset $\{X(t_1),\ldots,X(t_N)\}$ \\
m                             & State space dimension \\
N                             & Number of training samples \\
B                             & Closure of the trajectory set; assumed compact in $\mathbb{R}^m$ \\
U                             & Bounded, convex, open set with $B\subset U$ \\
h_{B,B_N}                     & Fill distance $\max_{x\in B}\min_{y\in B_N}\|x-y\|_2$ \\[6pt]

\multicolumn{2}{@{}l}{\textit{Dynamics and operators}} \\[2pt]
f:\mathbb{R}^m\to\mathbb{R}^m & True drift function of the autonomous ODE $\dot{X}=f(X)$ \\
L                           & Lipschitz constant of $f$ \\
M                             & Prediction horizon  \\
\Psi                          & Finite-difference operator approximating $\dot{X}(t_i)$ \\
\Phi                          & Numerical time-stepping function (e.g.\ RK45, symplectic) \\
\Delta t                      & Time step size for numerical integration \\

\multicolumn{2}{@{}l}{\textit{Gaussian process and kernel}} \\[2pt]
K(x,y)                        & Matérn kernel $k(\|x-y\|_2)=e^{-d}(1+d)$ \\
\mathrm{RKHS}(U)              & Reproducing Kernel Hilbert Space on $U$ induced by $K$; norm-equivalent to $H^{(m+3)/2}(U)$ \\
\mu(x)                        & GP posterior mean function; approximates $f$ \\
\sigma(x)                     & GP posterior standard deviation \\
\hat{\sigma}_N                & $\sup_{x\in B}\sigma_N(x)$; uniform bound on posterior std \\
b(x)                          & Lipschitz-corrected diffusion coefficient derived from $\sigma(x)$ \\
\tilde{f}(x)\sim\mathcal{N}(\mu(x),\sigma(x)^2 I_m) & GP surrogate for the drift function $f$ \\[6pt]

\multicolumn{2}{@{}l}{\textit{GPRODE model quantities}} \\[2pt]
E(t)                          & Predicted mean trajectory $\mathbb{E}_\omega[X(t,\omega)]$ \\
V(t)                          & Scalar variance proxy; $\Sigma(t)\approx V(t)I_m$ \\
\Sigma(t)                     & Approximate covariance matrix $\mathrm{Cov}_\omega(X(t,\omega))$ \\
\bar{V}                       & Threshold (saturation) variance \\
\gamma(x)                     & Saturation function $x/(x+0.001)$ enforcing $V\le\bar{V}$ \\
A(t)                          & Augmented state $[E(t)^\top, V(t)]^\top$ \\
g(A(t))                       & Right-hand side of the augmented ODE system \\[6pt]

\multicolumn{2}{@{}l}{\textit{QGPRODE --- reduction and latent space}} \\[2pt]
k                             & Latent (reduced) dimension \\
Z(t)\in\mathbb{S}^k           & Latent state on the unit sphere \\
Z_{\mathrm{train}}            & Latent codes of the training data \\
R(t)                          & Reduced coordinates before sphere projection \\
Q                             & Matrix such that $M^\star=Q^\top Q$; used for sphere projection \\
\bar{Q}                       & Low-rank approximation of $L$ \\
W_1\in\mathbb{R}^{m\times(k+1)} & Linear decoder matrix; $W_1=HL^+$ \\
\bar{W}_1                     & Low-rank version of $W_1$; $\bar{W}_1=H\bar{L}^+$ \\
W_2\in\mathbb{R}^{m\times(k+1)^2} & Quadratic decoder matrix \\
c                             & Constant offset vector in the quadratic reconstruction \\
\epsilon(t)                   & Reconstruction residual; modelled as white noise with variance $\sigma_\epsilon^2$ \\
v^{\otimes 2}                 & Kronecker self-product $v\otimes v$ \\
C_{ij}                        & Reduced covariance matrix $\sqrt{q_i q_j}\,X(t_i)\cdot X(t_j)$ \\
M^\star                       & Optimal positive definite matrix from sphere-fitting optimisation \\
\bar{E}[F]                    & Discrete time-average quadrature approximation of $(t_N-t_1)^{-1}\int F(t)\,dt$ \\
q_i                           & Quadrature weights (e.g.\ Monte Carlo: $q_i=1/N$) \\[6pt]

\multicolumn{2}{@{}l}{\textit{Moment equations (QGPRODE)}} \\[2pt]
m_k(t)                        & $k$-th centred moment $\mathbb{E}[(\tilde{Z}(t)-\mathbb{E}[\tilde{Z}(t)])^k]$ \\
m_3(t)                        & Third centred moment \\
m_4(t)                        & Fourth centred moment \\[6pt]

\bottomrule
\end{longtable}

\section{Methodology}
\label{sec:2}

We are interested in developing a model that approximates the behaviour of a real-world dynamical system based on observed data

\begin{equation}\label{eq:traindata}
X_{train}=\{X(t_1), \ldots, X(t_N)\}, \quad \text{with} \text{  }  X(t_i) \in \mathbb{R}^m \quad \forall i = 1, \ldots, N.
\end{equation}

In particular, our scope is to estimate the future states of the system under investigation, namely $X(t)$ for times $t > t_N$.
To achieve this objective, we develop two novel classes of ROMs. 

The proposed models aim 
to satisfy the following criteria:
\begin{itemize}
    \item High computational efficiency and rapid inference;
    \item A priori guarantee on computational complexity during both training and testing phases; 
    \item Uncertainty quantification associated with predictive outputs;
    \item Fast model update when new data is available.
\end{itemize}

\subsection{GPRODE model}
In this section, we introduce the definition of the GPRODE model.
We consider a set of training observations $X_{\text{train}}$ as in Eq.~\ref{eq:traindata}, in which the underlying dynamics are assumed to be governed by an ordinary differential equation of the form
\begin{equation}
\dot{X}(t) = f(X(t))
\end{equation}
and we furthermore assume that $X(t)$ is bounded.\\
To construct the model, we require three key components: a kernel function (thus a symmetric positive definite function with two arguments) $K:\mathbb{R}^{m}\times \mathbb{R}^{m} \rightarrow \mathbb{R}$, a finite-difference operator $\Psi$, and a time-stepping function $\Phi$. Common choices for $\Psi$ and $\Phi$ include the forward finite-difference operator and the forward Euler scheme, respectively. We begin by approximating the system dynamics  by means of $\Psi$, obtaining
\[
\dot{X}(t_i) \approx \Psi(X(t_i)).
\]

To simplify the notation, we adopt a vectorised formulation: for a collection of $N$ vectors $a$, each belonging to $ \mathbb{R}^{m}$, the expression $f(a)$ denotes the application of the function $f$ to each element of the set individually. For functions of two arguments, vectorisation is applied independently to each argument. Consequently, $f(X_{\text{train}})$ yields an $N \times m$ matrix, while $K(X_{\text{train}}, X_{\text{train}})$ produces an $N \times N$ matrix.

With this notation in place, we construct a reduced basis approximation of $f$, denoted by $\tilde{f}$, using Gaussian Process Regression \cite{rasmussen_gaussian_2005}, that is \begin{equation}
\tilde{f}(x)\sim \mathcal{N}(\mu(x),\sigma(x)^{2}I_{m}),
\end{equation}
 where
\begin{equation}
{\mu}(x)=K(x,X_{train})K(X_{train},X_{train})^{+}\Psi(X_{train})\in \mathbb{R}^{m},
\end{equation}
and
\begin{equation}
\sigma(x)=\sqrt{K(x,x)-K(x,X_{train})K(X_{train},X_{train})^{+}K(X_{train},x)}\in \mathbb{R},
\end{equation}
denote the mean and variance functions of the GP. We assume independent GPs for each component of 
$\Psi$, sharing the same kernel. The use of the pseudoinverse \(K(X_{\text{train}}, X_{\text{train}})^{+}\) is introduced to ensure that the resulting expression is well-defined, which may not occur in the presence of repeated inputs \cite{mohammadi_analytic_2017}. \\
By plugging in $\tilde{f}$ in place of $f$, we obtain the Random Differential Equation
\begin{equation}
        \dot{X}(t,\omega)=\tilde{f}(X(t,\omega),\omega).
    \end{equation}
However, the previous expression is expensive to simulate. In order to solve this issue, we
linearize $\tilde{f}$ around $E(t)=\mathbb{E}_{\omega}[X(t,\omega)]\in \mathbb{R}^{m}$ (assuming fluctuations $X(t)-E(t)$ are small), obtaining
\begin{equation}
        \dot{X}(t,\omega)=\tilde{f}(E(t),\omega)+\nabla_{x} \tilde{f}(E(t),\omega)(X(t)-E(t)).
    \end{equation}

Taking expectation and variance on both sides and assuming uncorrelation between $\nabla_{x}\tilde{f}(E(t))$ with $X(t)$ results in
\begin{equation}
\begin{cases}
\dot{E}(t) = \mu(E(t)),\\
\dot{\Sigma}(t) = \nabla_{x}\mu(E(t)) \Sigma(t)+\Sigma(t) \nabla_{x}\mu(E(t))^{\top}+\sigma^{2}(E(t))I_{m},
\end{cases}
\end{equation}

where $\Sigma(t)$ is an approximation of $\operatorname{Cov}_{\omega}(X(t,\omega))$.
A further simplification, obtained by neglecting derivatives, leads to

\begin{equation}
\begin{cases}
\dot{E}(t) = \mu(E(t)), \\
\dot{V} = \sigma^{2}(E(t))
\end{cases}
\end{equation}
where $V$ is a scalar such that $\Sigma=VI_{m}$. We apply this additional assumption, as while the Gaussian posterior mean converges to the true function under suitable hypotheses (specified in Appendix A), this may not hold for its derivative. This additional approximation also lowers the memory required to store $\Sigma(t)$ from $m^{2}$ to $1$, which is crucial since we are dealing with large-scale dynamical systems. \\
We note that the above formulation can also be obtained by doing a 0th-order approximation around the mean of the Stochastic Differential Equation 
\begin{equation}
    \begin{cases} 
d\tilde{X}=\mu(\tilde{X}(t))dt+\sigma(\tilde{X}(t))I_{m}dW & t\in (t_{N},t_{N}+M],\\
\tilde{X}(t_{N})={X}(t_{N}).
\end{cases}
\end{equation}

We note that while $\mu(x)$ is a Lipschitz function, $\sigma(x)$ is not globally Lipschitz as a result of the presence of the square root term. However, since we intend to invoke the Lipschitz condition in order to ensure uniqueness of the solution to the SDE, we introduce a slightly modified equation in which the diffusion coefficient \(\sigma\) is replaced by a Lipschitz continuous approximation \(b \in C^{1}(U)\) defined by
\[
    b(x)=
    \begin{cases}
    \sigma(x), & \sigma(x) \ge 1,\\[4pt]
    \dfrac{2 \sigma(x)^{2}}{1+\sigma(x)^{2}}, & \sigma(x)<1,
    \end{cases}
\]
which yields the stochastic differential equation
\[
\begin{cases} 
d\tilde{X}(t)=\mu(\tilde{X}(t))\,dt + b(\tilde{X}(t)) I_{m}\,dW(t), & t\in (t_{N},t_{N}+M],\\
\tilde{X}(t_{N}) = {X}(t_{N}).
\end{cases}
\]

Since $b$ is strictly positive, the function $V(t)$ is strictly increasing in time. Moreover, as the system under consideration is bounded, there exists a threshold variance $\bar{V}\in \mathbb{R}$ beyond which further increases in variance do not yield additional information. Consequently, we introduce the following modification to enforce the constraint $V(t)\le \bar{V}$:
\begin{equation}\label{eq:gprode}
\begin{cases}
\dot{E}(t)=\mu(E(t)), & t\in (t_{N},t_{N}+M],\\[4pt]
\dot{V}(t)=b(E(t))^{2}\,\gamma\!\left(\dfrac{\bar{V}-V(t)}{\bar{V}}\right), & t\in (t_{N},t_{N}+M],\\[4pt]
E(t_{N})=X(t_{N}),\\[2pt]
V(t_{N})=0,
\end{cases}
\end{equation}
where
\begin{equation*}
\gamma(x)=\frac{x}{x+0.001}
\end{equation*}
which enforces saturation as $V\rightarrow \bar{V}$.
A choice for the threshold variance can be the ergodic variance, provided the system is ergodic. The selection of $\gamma$ is not restrictive and can be changed with any decreasing function depending only on $V$ that has a zero in $\bar{V}$.
Now we want to write the modified ODE system in a more readable form.
Let us consider 
\begin{equation}
\begin{gathered}
{A}(t)=\left[\begin{array}{c}
{E}(t) \\
{V}(t)
\end{array}\right], \\
A_{0}=\left[\begin{array}{c}
X(t_{N}) \\
0
\end{array}\right], \\
{g}\left({A}(t)\right)=\left[\begin{array}{c}
{\mu}\left({E}(t)\right) \\
b^2\left({E}(t)\right) \gamma\left(\frac{\bar{V}-{V}(t)}{\bar{V}}\right)
\end{array}\right].
\end{gathered}
\end{equation}

This allows us to obtain the following expression
$$
\left\{\begin{array}{l}
{A}(t_{N})=A_{0}, \\
\dot{{A}}={g}\left({A}\right) .
\end{array}\right.
$$

In Appendix A, we prove that all equations considered converge to the underlying ODE under suitable assumptions on the data and on $f$. In particular, we assume that the data is infinitely dense, that f is smooth enough, and that the trajectory is bounded.

To numerically solve the ODE above, we employ the time-stepping operator $\Phi$ with step size $\Delta t.$
The resulting discrete formulation is
\begin{equation}
A_{k}=\Phi\left(A_{k-1}, {g}, \Delta {t}\right),
\end{equation}
where $\Phi(A_{k-1}, {g}, \Delta {t})$ advances the solution by one time step from $A_{k-1}$, using $g$ as the right-hand side of the ODE over the interval of length $\Delta t$.

We summarise the procedure in Algorithm \ref{alg:gprsde}.

\begin{algorithm}[H]
\caption{GPRODE Training and Prediction}
\label{alg:gprsde}
\begin{algorithmic}[1]
\State \textbf{Input:} Training data $\{X(t_1), \ldots, X(t_N)\}$, kernel $K$, prediction horizon $M$, target variance $\bar{V}$, time step $\Delta t$, finite-difference operator $\Psi$, time-stepping function $\Phi$.
\State \textbf{Output:} Mean trajectory $E(t)$, variance $V(t)$ for $t \in [t_N, t_N + M].$
\State // \textit{Step 1: Train Gaussian process on approximated derivatives}
\State Assemble the Gram Matrix $\bar{K}_{ij}=K(X(t_i),X(t_j))$ 
\State Define mean function: $\mu(x) = K(x, X_{\text{train}})\bar{K}^{-1}\Psi(X_{train})$
\State Define variance function: $\sigma^2(x) = K(x,x) - K(x, X_{\text{train}})\bar{K}^{+}K(X_{\text{train}}, x)$
\State Define corrected variance function $b(x)=\begin{cases}
    \sigma(x) & \sigma(x) \ge 1\\
    \frac{2 \sigma(x)^{2}}{1+\sigma(x)^2} & \sigma(x)<1
\end{cases}$
\State // \textit{Step 2: Propagate mean and variance forward}
\State Initialize: $E(t_{N}) = X(t_N)$, $V(t_N) = 0$

\State Assemble
\begin{align*}
\begin{gathered}
{A}(t)=\left[\begin{array}{c}
{E}(t) \\
{V}(t)
\end{array}\right], \\
A_{0}=\left[\begin{array}{c}
X(t_{N}) \\
0
\end{array}\right], \\
{g}\left({A}(t)\right)=\left[\begin{array}{c}
{\mu}\left({E}(t)\right) \\
b^2\left({E}(t)\right) \gamma\left(\frac{\bar{V}-{V}(t)}{\bar{V}}\right)
\end{array}\right],
\end{gathered}
\end{align*}
where 
$\gamma(x)=\frac{x}{x+0.001}$

\State Solve coupled ODEs using numerical integrator $\Phi$:

\For{each $i$ from 1 to $\lceil{\frac{M}{\Delta t}}\rceil$}
\State $A_{i}=\Phi(A_{i-1},g,\Delta t)$
\EndFor

\State Extract $E_{i},V_{i}$ from $A_{i}$
\State \textbf{Return} $E_{i},V_{i},i=1..\lceil{\frac{M}{\Delta t}}\rceil$
\end{algorithmic}
\end{algorithm}

In Appendix A, we demonstrate that this numerical method is convergent in the limit of infinite data. In the next section, we confirm these theoretical results through numerical experiments.\\
In the context of Digital Shadows, when a new data point becomes available, the earliest element in the training set is removed so that the overall training size remains constant. In this setting, this framework respects the criteria given at the start of Section \ref{sec:2}:
\begin{itemize}
    \item It is computationally efficient, as it scales linearly in the data dimensionality,    
    \item An a priori bound on computational cost is $mN^{2}+N^{3}$ during training and $\lceil \frac{M}{\Delta t}\rceil N^2$ during testing.
    \item The variance $V$ represents the uncertainty of each prediction,
    \item Updating the model is fast, as it is sufficient to retrain it after each data update.
\end{itemize}
Conversely, the principal limitation of this methodology, relative to other Gaussian Process–based ODE frameworks, is that the variance is designed to increase monotonically over time. As a consequence, the model frequently yields uncertainty bounds that are more conservative than warranted. This represents a necessary trade-off to enable high computational efficiency.\    

\subsubsection{Numerical verification}
We now examine the numerical convergence behaviour of the method. In practical applications, the infinite-data regime is unattainable, since Gaussian Processes scale cubically with the number of data points due to the need to solve a dense linear system. As in the appendix, we assume that $B=\overline{\{X(t)\mid t>0\}}$ is a compact subset of $\mathbb{R}^{m}$, and that there exists a bounded, convex, open set $U$ such that $B \subset U$ and $f \in H_{0}^{\frac{m+5}{2}}(U,\mathbb{R}^{m})$.  
We denote by $RKHS(U)$ the Reproducing Kernel Hilbert Space \cite{wendland_scattered_2004} on $U$ induced by the kernel $K$. Choosing $K$ as
\begin{equation}
 K(x,y)=k(||x-y||_{2}),
\end{equation}
where
\begin{equation}
    k(d)=e^{-{d}}(1+d),
\end{equation}
it can be proved that $RKHS(U)$ is norm equivalent to $H^{\frac{3+m}{2}}(U)$ \cite{wendland_scattered_2004}.

Regarding $\Psi$, we construct it as a continuous function, assuming that the training points correspond to $X(0), X(\frac{1}{\sqrt{N}}),...X(\sqrt{N})$
\begin{equation}
    \Psi(x):=\frac{y\left(\frac{1}{\sqrt{N}}, x\right)-x}{\frac{1}{\sqrt{N}}},
\end{equation}
where $y(t, x)$ is the flow of the ODE.
This construction coincides with the forward difference approximation on the training data, while ensuring sufficient regularity for $\Psi \in R K H S([a, b])$ for sufficiently large $N$, as it is a first order approximation of a function in $f \in H_{0}^{\frac{m+5}{2}}(U,\mathbb{R}^{m})$ (as formally proven in Appendix A).
This definition only works because the time is equispaced with time step $\frac{1}{\sqrt{N}}$. In the case of non-uniform temporal discretisation, the construction of $\Psi$ becomes highly nontrivial, since it must depend solely on $x$ and, in addition, possess sufficient regularity to lie in the space $H^{\frac{m+5}{2}}(U,\mathbb{R}^{m})$. An alternative approach is shown in Appendix A.
As $\Phi$ we choose the explicit Euler method with $\Delta t=\frac{1}{\sqrt{N}}$.\\

With these ingredients, it can be proven that the testing error decreases and satisfies the bound (as shown in Appendix A)
\begin{equation}
\begin{gathered}
\left\|X(\sqrt{N}+T)-E_{\lceil \sqrt{N}T \rceil} \right\|_{2} \leq \\\frac{e^{L T}-1}{L}\left[||\Psi||_{RKHS(U)}\left(1-k\left(\frac{1}{\sqrt{N}}\right)\right)+(L \sup\limits_{U}||f(x)||_{2}) \left(\frac{1}{\sqrt{N}}\right)\right],
\end{gathered}
\end{equation}
where  $L=\sup\limits_{x\in U}||\nabla f(x)||_{2}$. Estimating the theoretical bound requires knowledge of $||\Psi||_{RKHS(U)}$, which cannot be computed analytically in the $U$ bounded case. An exception is the case $m=1$, for which $U=[a,b]$ for certain $a$ and $b$, and it holds \cite{saturo_extension_2024}

\begin{equation}
\begin{gathered}
||\Psi||_{RKHS([a,b])}^{2}=
 \frac{1}{4} \left\|  \Psi^{\prime \prime}\right\|_{L^2([a, b])}^2+\frac{1}{2}\left\|\Psi^{\prime}\right\|_{L^2([a, b])}^2 \\ +\frac{1}{4}\|\Psi\|_{L^2([a, b])}^2
 +\frac{1}{2}\left(\Psi(a)^2-\Psi(a) \Psi^{\prime}(a)+\Psi^{\prime}(a)^2\right)\\ +\frac{1}{2}\left(\Psi(b)^2-\Psi(b) \Psi^{\prime}(b)+\Psi^{\prime}(b)^2\right) .
 \end{gathered}
\end{equation}
Thus, for assessing the performance of the numerical method, we restrict to $m=1$, and $U=[a,b].$\\

The kernel mean and variance functions are calculated using the Scikit-Learn library \cite{pedregosa_scikit-learn_2011}, which has ad-hoc adjustments to compute an accurate solution of the linear system, including adding a jitter. $X_train$ is normalised before training the model. This approach for calculating the mean and variance functions will be used throughout the remainder of the paper. Since the kernel is parameter-free, no hyperparameter optimisation is performed, and this convention is likewise followed in the rest of the paper. \\

We study the convergence of two different ODEs, and we compare them with the respective theoretical bounds:
\begin{itemize}
    \item $y'=-y, y(0)=\frac{1}{2}$,
    \item $y'=-y(1-y), y(0)=\frac{1}{2}$.
\end{itemize}
Results are reported in Fig. \ref{fig:lin}. As can be observed, the numerical errors remain within the theoretical error bound in both cases, and the convergence is substantially faster than that predicted by the bound.

\begin{figure}
\centering
\includegraphics[width=1.0\textwidth]{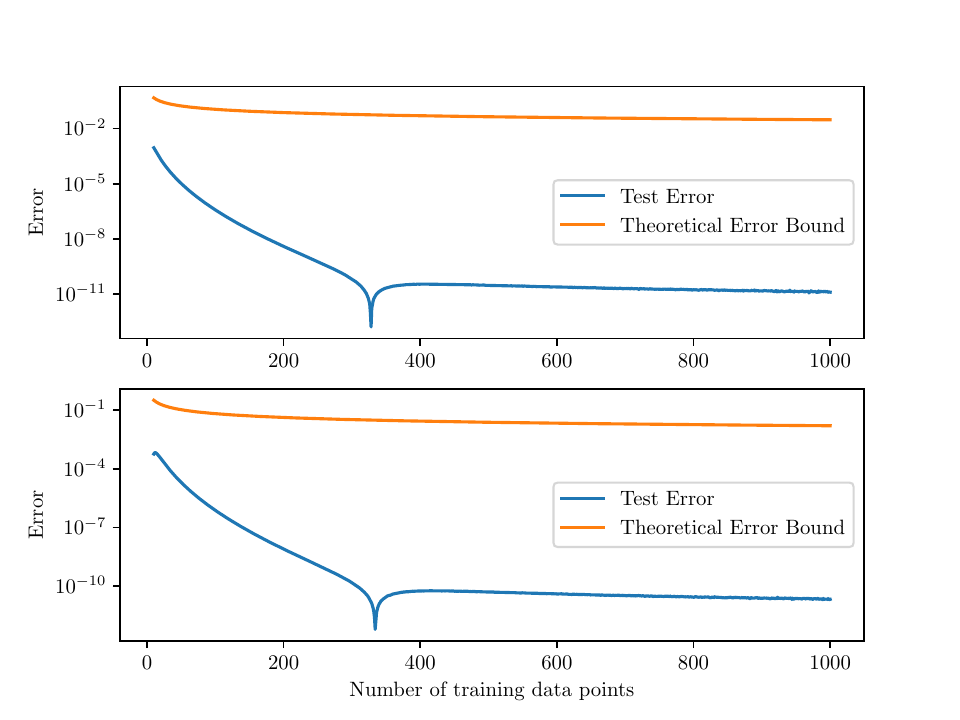}
\caption{Top: theoretical error bound and numerical error for $y'=-y$. Bottom: theoretical error bound and numerical error for $y'=-y(1-y)$. There is a cusp due to numerical instability caused by the growing condition number of the kernel Gram matrix.}
\label{fig:lin}    
\end{figure}

\subsection{QGPRODE}
In Appendix A, it is proved that the GPRODE model converges in the limit of infinite data.\\
Nevertheless, we observe that, in general, the simulated trajectory may exhibit non-physical behaviour due to over-fitting (this is particularly relevant because $N$ must remain small, given the $O(N^{3})$ scaling in the computation of $\mu$). This issue can be addressed by constructing a reduced-order model in space.\\
We will address this problem by projecting the equation onto a $k$-dimensional unit sphere and thus performing dimensionality reduction.
Motivated by the work on Quadratic Model Order Reduction \cite{geelen_operator_2023}, our aim is to obtain a decomposition of the form\begin{equation}
X(t)=c+W_{1}Z(t)+W_{2} \operatorname{vec}(Z(t)\otimes Z(t))+\epsilon(t),
\end{equation}
where $Z(t)\in \mathbb{S}^{k}$ is low-dimensional, $W_{1}\in R^{m\times (k+1)}$ and $W_{2}\in R^{m\times (k+1)^{2}}$ are matrices, $c$ is a constant vector, and $\epsilon(t)$ is a residual.\\
We first introduce
\[
\bar{E}[F]: C([t_{1},t_{n}],\mathbb{R}^{k})\rightarrow \mathbb{R}^{k}
\]
by setting
\begin{equation}
\bar{E}[F]=\frac{1}{t_{N}-t_{1}}\Xi[F(\cdot),t_{1},t_{N}],
\end{equation}
where \(\Xi\) denotes a quadrature rule that has \(X_{\text{train}}\) as quadrature points, and is used to approximate
\begin{equation}
\int_{t_{1}}^{t_{N}}F(t)dt.
\end{equation}

Then we can compute $c_{1}=\bar{E}[X(\cdot)]$ and 
\[
C_{1}=\bar{E}\big[(X(\cdot)-c_1)\otimes(X(\cdot)-c_1)\big]\in \mathbb{R}^{m\times m}.
\]
Since $C_{1}$ is symmetric and positive semidefinite, we can write $C_{1}=HH^{T}$ thanks to the Singular Value Decomposition (SVD). $H$ is the product of the matrix with columns containing the eigenvectors by the square root of the corresponding eigenvalues, sorted in descending order.  \\
Next, we define $R(t)$ as the vector consisting of the first $k+1$ components of $H^{+}(X(t)-c_{1})$.
The problem of this approach is that $C_{1}$ cannot be stored in memory due to $m$ being very large.\\
So instead we compute the reduced matrix $C_{ij} = \sqrt{q_{i}q_{j}}X(t_{i})\cdot X(t_{j})$ and $R(t_i)$ is the $i$-th left eigenvector of $C$ truncated at index $k+1$ via randomised SVD. The $q_{i}$ are the quadrature weights of $\Xi$.\\
We then solve the optimisation problem
\[
M^{\star}=\min_{M \in \mathbb{R}^{k+1\times k+1},\, M>0 } \bar{E}\big[(R(\cdot )^{T}MR(\cdot)-1)^{2}\big]
+ a\big(\lambda_{\max}(M)-\lambda_{\min}(M)\big),
\]
where $a$ is a regularisation parameter.\\
Subsequently, we use the spectral decomposition to find $Q$ such that $M^{\star}=Q^{T}Q$.\\
The latent quantity $QR(t)$ will then lie approximately in $\mathbb{S}^{k}$; to enforce that it lies on $\mathbb{S}^{k}$, we set
\begin{equation}
Z(t)=\frac{QR(t)}{\|QR(t)\|}.
\end{equation}
We also store a low-rank approximation of $Q$, denoted by $\bar{Q}$.
We then set $W_{1}=HQ^{+}$ and $\bar{W}_{1}=H\bar{Q}^{+}$.
Next, we obtain $W_{2}$ by solving
\begin{equation}
\min_{W_{2}}\,\bar{E}[\bigl\|X(\cdot)-W_{1}Z(\cdot)-W_{2}\bigl(\operatorname{vec}(Z(\cdot)^{\otimes 2})\bigr)\bigr\|^{2}].
\end{equation}

where
\begin{equation}
v^{\otimes 2}=v\otimes v.
\end{equation}
If $k^{2}$ is large, $Z(t)$ is truncated only in the quadratic term.\\
We then define
\begin{equation}
    c=\bar{E}[X(\cdot)]-W_{1}\bar{E}[Z(\cdot)]-W_{2}\bar{E}[\operatorname{vec}(Z(\cdot)^{\otimes 2})].
\end{equation}

This yields the representation
\begin{equation}\label{eq:qogprsde}
X(t)=c+W_{1}Z(t)+W_{2}\bigl(\operatorname{vec}(Z(t)^{\otimes 2})\bigr)+\epsilon(t).
\end{equation}
By construction, it holds that 
\begin{equation}
\bar{E}[\epsilon(\cdot)]=0.
\end{equation}

A natural choice to model $\epsilon(t)$ is as a white noise process of variance 

$\bar{E}[\epsilon(\cdot)^{2}]$.

During inference, $\tilde{Z}(t)$ is predicted using the methodology of the previous section: i.e., defining
\begin{equation}
{\mu}(z)=K(z,Z_{train})K(Z_{train},Z_{train})^{+}\Psi(Z_{train}),
\end{equation}
\begin{equation}
\sigma(z)=\sqrt{K(z,z)-K(z,Z_{train})(K(Z_{train},Z_{train}))^{+}K(Z_{train},z)}.
\end{equation}
\begin{equation}
    b(z)=
    \begin{cases}
    \sigma(z) & \sigma(z) \ge 1,\\
    \frac{2 \sigma(z)^{2}}{1+\sigma(z)^{2}} & \sigma(z)<1,
\end{cases}   
\end{equation}
where $Z_{train}$ are the latent codes of $X_{train}$.\\
The values $E(t)=\mathbb{E}[\tilde{Z}(t)]$ and $V(t)=\mathbb{E}[(\tilde{Z}(t)-\mathbb{E}[\tilde{Z}(t)])^{2}]$ are computed with
\begin{equation}
\begin{cases}
\dot{{E}}(t)=\mu({E}(t))(I-\frac{E(t)E(t)^{T}}{||E(t)||_{2}^{2}}) & t\in (t_{N},t_{N}+M],\\ 
\dot{V}(t)=b({E}(t))^{2}\gamma\left(\frac{\bar{V}-V(t)}{\bar{V}}\right) & t\in (t_{N},t_{N}+M],\\ 
{E}(t_{N})=Z(t_{N}),\\
{V}(t_{N})=0.\\
\end{cases}
\end{equation}
With respect to Eq. \ref{eq:gprode}, the equation for $E(t)$ is modified in order to keep $E(t)$ on the unit sphere, thus enforcing boundness. To guarantee that the numerical solution remains confined to the unit sphere, we additionally employ a symplectic integrator as the numerical integration scheme.

The distribution of $\tilde{X}(t)$ is then computed using
\begin{equation}
\tilde{X}(t)=c+W_{1}\tilde{Z}(t)+W_{2}\operatorname{vec}(\tilde{Z}(t)^{\otimes 2})+\epsilon(t).
\end{equation}

In order to obtain the mean and variance for the corresponding predictor of $\tilde{X}(t)$, due to the presence of $\tilde{Z}(t) \otimes \tilde{Z}(t)$, we also need the third and fourth centred moments of the random variable given by the equation
\begin{equation}    
\begin{cases} 
d\tilde{X}=\mu(\tilde{X}(t))dt+b(\tilde{X}(t))I_{m}dW & t\in (t_{N},t_{N}+M],\\
\tilde{X}(0)=\tilde{X}(t_{N}).
\end{cases}
\end{equation}

By assuming a 0-order approximation for $\mu_{N}$ and $\sigma_{N}$ around $E[\tilde{Z}(t)]$, the moments can be approximated as 
\begin{equation}
\frac{dm_{3}}{dt}=3V(t)\mu(E(t)),
\end{equation}
\begin{equation}
\frac{dm_{4}}{dt}=4m_{3}(t)\mu(E(t))+6 V(t)b^{2}(E(t)),
\end{equation}
where  $m_{k}(t)=E[(\tilde{Z}(t)-E[\tilde{Z}(t)])^{k}]$.
When computing the variance of $\tilde{X}(t)$, we use the low-rank version of $W_{1}$, $\bar{W}_{1}$, to avoid variance explosion due to the fact that $L$ may have eigenvalues near 0, which could potentially cause numerical problems.\\
We modify the equation to make sure that the fourth moment is non-negative and bounded in the following way
\begin{equation}
\frac{dm_{4}}{dt}=(\max(4m_{3}(t)\mu(E(t)),0)+6 V(t)b^{2}(E(t)))\odot
\gamma\left(\frac{\bar{V}^{2}-m_{4}}{\bar{V}^{2}}\right). 
\end{equation}

As $\bar{V}$ we adopt the ergodic variance 
\begin{equation}
\bar{E}[(Z(\cdot)-\bar{E}[Z(\cdot)])^{2}].
\end{equation}
Then $\mathbb{E}[X(t)]$ can be computed as
\begin{equation}
\mathbb{E}[\tilde{X}(t)] = c + W_1 E(t) + W_2(\operatorname{vec}(E(t)^{\otimes 2}+diag(V(t))))
\end{equation}
while
\begin{equation}
\begin{gathered}
Var[\tilde{X}(t)] = \bar{W}_1 {diag}(V(t)) \bar{W}_1^T 
+ W_2 \mathbb{E}[((\operatorname{vec}(\tilde{Z}(t)^{\otimes 2} - E(t)^{\otimes 2}))^{\otimes 2})] W_2^T \\
+ \sigma_{\epsilon}^2 I.
\end{gathered}
\end{equation}
Details on the computation of $[((\operatorname{vec}(\tilde{Z}(t)^{\otimes 2} - E(t)^{\otimes 2}))^{\otimes 2})]$ can be found in Appendix C.
With respect to the GPRODE model, the QGPRODE model is more efficient, as it has to compute a randomised SVD at a cost of $O\left(m n \log (k)+(m+N) k^2\right)$, then $O(N^{2}k)$
for assembling the Gram matrix and solving the linear system, $O(k)$ for predicting, and then $O(mk^{2})$ for predicting, while the GPRODE costs $O(mN^{2})$. 
We summarise the methodology in Algorithm 2. A summary of the algorithm steps is shown in Fig. \ref{fig:workflow}.
\begin{algorithm}[H]
\caption{QGPRODE Training and Prediction}
\label{alg:QGPRODE}
\begin{algorithmic}[1]
\State \textbf{Input:} Training data $\{X(t_1), \ldots, X(t_N)\}$, kernel $K$, prediction horizon $M$, target variance $\bar{V}$, time step $\Delta t$, finite-difference operator $\Psi$, time-stepping symplectic function $\Phi$, latent dimension $k$, regularization parameter $a$, quadratic truncation dimension $k_{\text{quad}}$ (optional). Quadrature weights $q_{1}$...$q_{N}$.
\State \textbf{Output:} Mean trajectory $\mathbb{E}[\tilde{X}(t)]$, variance $\text{Var}[\tilde{X}(t)]$ for $t \in [t_N, t_N + M]$.
\State // \textit{Step 1: Compute quadratic reduced-order model}
\State Define $\bar{E}[f(\cdot)]=\sum_{i=1}^{n} q_{i}f(t_{i})$
\State Compute mean: $c_1 =\bar{E}[(X(\cdot)] $
\State Compute the reduced matrix $C_{ij} = \sqrt{q_{i}q_{j}}X(t_{i})\cdot X(t_{j})$
\State Compute reduced coordinates: $R(t_i)$ is the $i$-th left eigenvector of $C$ truncated at index $k+1$ via Randomized SVD  
\State Solve optimisation problem:
\[
M^{\star} = \arg\min_{M \in \mathbb{R}^{k+1 \times k+1}, M > 0} \bar{E}[(R(\cdot)^T M R(\cdot) - 1)^2] + a(\lambda_{\max}(M) - \lambda_{\min}(M))
\]
\State Compute: $M^{\star} = Q^T Q$ (via spectral decomposition)
\State Compute latent codes: $Z(t_i) = \frac{LR(t_i)}{\|LR(t_i)\|}$ for $i = 1, \ldots, N$
\State Store: $Z_{\text{train}} = \{Z(t_1), \ldots, Z(t_N)\}$
\State // \textit{Step 2: Compute GPRODE on } $Z(t)$
\State Compute low-rank version: $\bar{Q}$ (truncated version of $Q$)
\State Define: $W_1 = HQ^+$ and $\bar{W}_1 = H\bar{Q}^+$
\State Compute $W_2$ by minimizing: $\bar{E}[\|X(\cdot) - W_1 Z(\cdot) - W_2\operatorname{vec}(Z(\cdot)^{\otimes 2}  )\|^2]$
\State \textit{(Optional: If $k^2$ is large, truncate $Z$ to dimension $k_{\text{quad}}$ for quadratic term)}
\State Compute offset: $c = \bar{E}[X(\cdot)] - W_1 \bar{E}[Z(\cdot)] - W_2 \bar{E}[\operatorname{vec}(Z(\cdot)^{\otimes 2})]$
\State Compute noise variance: $\sigma_{\epsilon}^2 = \bar{E}[\|X(\cdot) - c - W_1 Z(\cdot) - W_2(\operatorname{vec}(Z(\cdot)^{\otimes 2}))\|^2]$

\State Compute kernel matrix: $\bar{K}_{ij} = K(Z(t_i), Z(t_j))$ for $i,j = 1, \ldots, N$
\State Define mean function: $\mu(z) = K(z, Z_{\text{train}})\bar{K}^{+}\Psi(Z_{\text{train}})$
\State Define variance function: $\sigma^2(z) = K(z,z) - K(z, Z_{\text{train}})\bar{K}^{+}K(Z_{\text{train}}, z)$
\State Define corrected variance: $b(z) = \begin{cases}
    \sigma(z) & \text{if } \sigma(z) \ge 1\\
    \frac{2\sigma(z)^2}{1 + \sigma(z)^2} & \text{if } \sigma(z) < 1
\end{cases}$
\algstore{myalg}
\end{algorithmic}
\end{algorithm}

\begin{algorithm}[H]         
\begin{algorithmic} [1]                   % enter the algorithmic environment
\algrestore{myalg}
\State Define $\gamma(x) = \frac{x}{x + 0.001}$
\State Assemble augmented state and dynamics:
\begin{align*}
A(t) &= [E(t)^T, V(t)^T, m_3(t)^T, m_4(t)^T]^T \\
A_0 &= [Z(t_N)^T, 0, 0, 0]^T \\
g(A(t)) &= \begin{bmatrix}
\mu(E(t))  \left(I - \frac{E(t)E(t)^T}{\|E(t)\|_2^2}\right) \\
b^2(E(t)) \gamma\left(\frac{\bar{V} - V(t)}{\bar{V}}\right) \\
3V(t) \odot \mu(E(t)) \\
(\max(4m_3(t) \odot \mu(E(t)), 0) + 6V(t) \odot b^2(E(t))) \odot \gamma\left(\frac{\bar{V}^2 - m_4(t)}{\bar{V}^2}\right)
\end{bmatrix}
\end{align*}
where $\odot$ denotes element-wise multiplication

\For{each $i$ from 1 to $\lceil\frac{M}{\Delta t}\rceil$}
\State $A_i = \Phi(A_{i-1}, g, \Delta t)$
\State Extract: $E_i$, $V_i$, $m_{3,i}$, $m_{4,i}$ from $A_i$
\EndFor
\State
\For{each $i$ from 1 to $\lceil\frac{M}{\Delta t}\rceil$}
\State Compute mean: $\mathbb{E}[\tilde{X}(t_N + i\Delta t)] = c + W_1 E_i + W_2\operatorname{vec}(E_i^{\otimes 2}+diag(V_{i}))$
\State Compute variance using low-rank $\bar{W}_1$:
\begin{align*}
\text{Var}[\tilde{X}(t_N + i\Delta t)] = &\bar{W}_1 \text{diag}(V_i) \bar{W}_1^T \\
&+ W_2 \mathbb{E}[(\operatorname{vec}(Z^{\otimes 2} - E_i^{\otimes 2}))^{\otimes 2}] W_2^T \\
&+ \sigma_{\epsilon}^2 I
\end{align*}
where $\mathbb{E}[(\operatorname{vec}(Z^{\otimes 2} - E_i^{\otimes 2}))^{\otimes 2}]$ is computed using $V_i$, $m_{3,i}$, $m_{4,i}$
\EndFor

\State \textbf{Return} $E[\tilde{X}(t)]$ and $\text{Var}[\tilde{X}(t)]$ for $t \in [t_N, t_N + M]$
\end{algorithmic}
\end{algorithm}

\begin{figure}
    \centering
    \includegraphics[width=\linewidth]{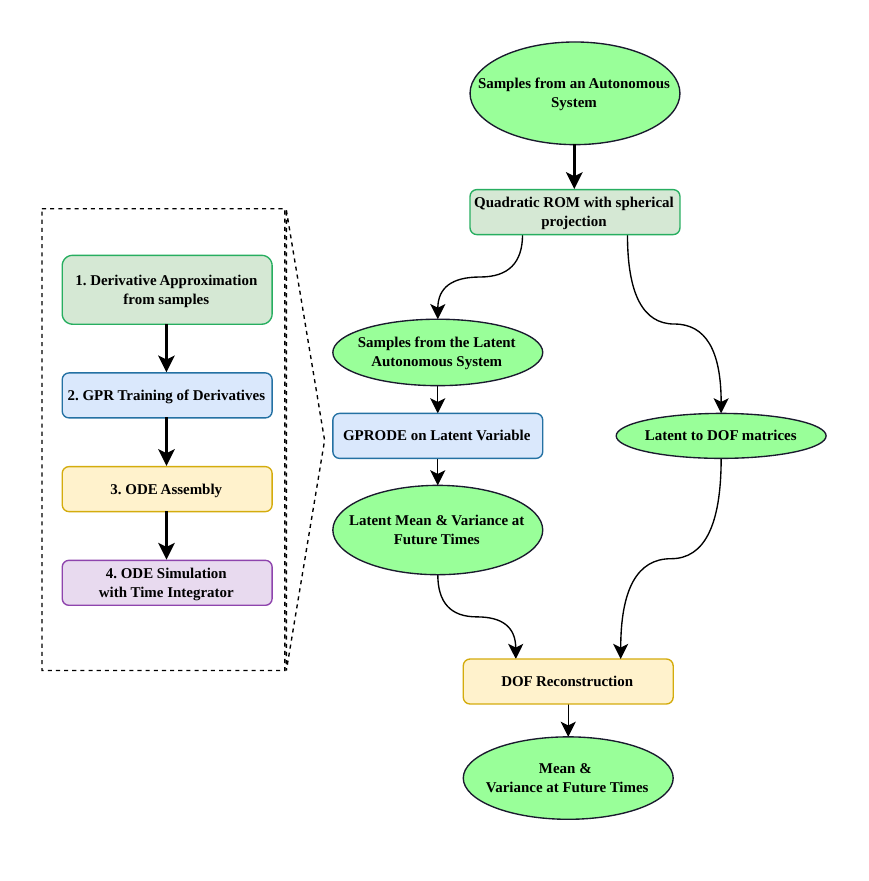}
    \caption{The workflow of the QGPRODE algorithm.}
    \label{fig:workflow}
\end{figure}

\section{Applications}
In this section, we provide numerical experiments to validate our methodology. All of them are performed on a machine with an Intel(R) Xeon(R) Platinum 8260 CPU @ 2.40GHz. Computational times reported are the average of 1000 runs.
Before investigating more advanced test cases, we perform a sanity check for the GPRODE model by considering a Lorenz system, which reads as follows: 
\begin{equation}
\label{eq:Lorenz}
\begin{cases}
\frac{du}{dt}=-\sigma(u-v) & t \in [0,20],\\
\frac{dv}{dt}=\rho u -v -uw & t \in [0,20],\\
\frac{dw}{dt}=-\beta w+u v & t \in [0,20],\\
u(0)=\frac{1}{10},   \\
v(0)=0, \\
w(0)=0, \\
\end{cases}
\end{equation}
with $\sigma=10$, $\beta=\frac{8}{3}$, $\rho=13$.
With these parameters, the system converges to an equilibrium point, as shown in Fig. \ref{fig:0}. As we can observe from Fig. \ref{fig:1}, the GPRODE is able to reproduce the system dynamics. The database consists of $20001$ time steps, of which the first $10000$ are used for training. 
For derivative estimation (even if it is known in this case), we adopt the forward difference operator as $\Psi$. For simulation, we adopt an explicit RK45 scheme as $\Phi$.\\
We compare the GPRODE with  LANDO, BOPDMD and EDMD from the library PyDMD \cite{demo_pydmd_2018,ichinaga_pydmd_2024} with rank 3 and default parameters. We also measure the performance of QGPRODE (with rank 3) for completeness. 
To measure forecasting accuracy, we adopt the scaled L1 and relative L2 error. Let $Y\in \mathbb{R}^{N\times m}$ the true values and $\tilde{Y}\in \mathbb{R}^{N\times m}$ the values as predicted by the model. We define the scaled L1 error as
\begin{equation}
\frac{1}{\max\limits_{i=1\ldots N}\max\limits_{j=1\ldots m} Y_{ij}-\min\limits_{i=1\ldots N}\min\limits_{j=1\ldots m} Y_{ij}}
\frac{1}{mN}\sum\limits_{i=1}^{N}\sum\limits_{j=1}^{m}|Y_{ij}-\tilde{Y}_{ij}|,
\end{equation}
and the relative L2 error as
\begin{equation}
\frac{\sqrt{\sum\limits_{i=1}^{N}\sum\limits_{j=1}^{m}(Y_{ij}-\tilde{Y}_{ij})^{2}}}{\sqrt{\sum\limits_{i=1}^{N}\sum\limits_{j=1}^{m}Y_{ij}^{2}}}.
\end{equation}

As shown in Table \ref{tab:1}, GPRODE is able to outperform existing methodologies, either in terms of accuracy or efficiency. With respect to LANDO, the required time is substantially shorter because LANDO must subsample the data to ensure stability, which is computationally expensive, whereas our approach achieves stability through the use of time integrators.
EDMD is slower than our method as it must first compute two kernel matrices and then a Singular Value Decomposition for truncation. However, because it is trained on a relatively small number of samples, it tends to overfit the training data, which in turn leads to a high test error.  
BOPDMD, being a linear predictor, is less susceptible to overfitting than EDMD and is the fastest because no kernel matrix must be computed. Nevertheless, since the underlying system is nonlinear, it yields a higher error than both LANDO and GPRODE. Because there is no truncation, the latent system of QGPRODE differs from the GPRODE one by a stretch and a rotation; thus, QGPRODE is also exact in the infinite data limit and thus accurate as the hypotheses are satisfied.
\begin{table}[H]
\centering
\begin{tabular}{|l|l|l|l|}
\hline
Type     & scaled L1   & relative L2   & Time (seconds)\\ \hline
EDMD     & 1.21 & 1.29 & 211 \\ \hline
BOPDMD   & 0.45 & 0.48 & \textbf{2} \\ \hline
GPRODE   & \textbf{2.22e-04}  & \textbf{2.23e-04} & 39 \\ \hline
QGPRODE   & {2.31e-04}  & {3.39e-04} & 45 \\ \hline

LANDO    & 2.29e-04 & 3.32e-04 & 314 \\ \hline
\end{tabular}
\caption{Lorentz dynamical system. Scaled L1 and relative L2 error between each predictor and the true value.}
\label{tab:1}
\end{table}

\begin{figure}
\centering
\includegraphics[width=\textwidth]{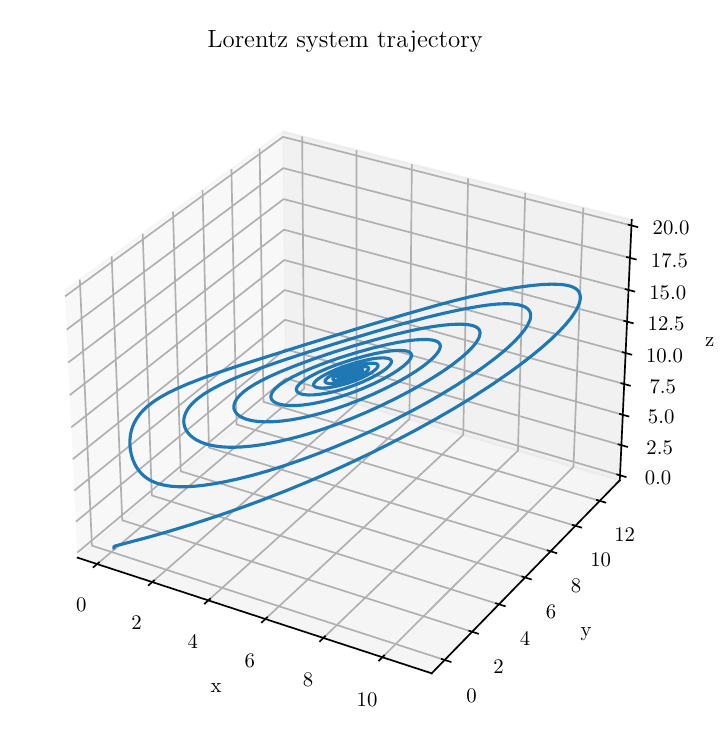}
\caption{Trajectory of the Lorentz system given by Eq. \ref{eq:Lorenz}.}
\label{fig:0}    
\end{figure}

\begin{figure}
\centering
\includegraphics[width=\textwidth]{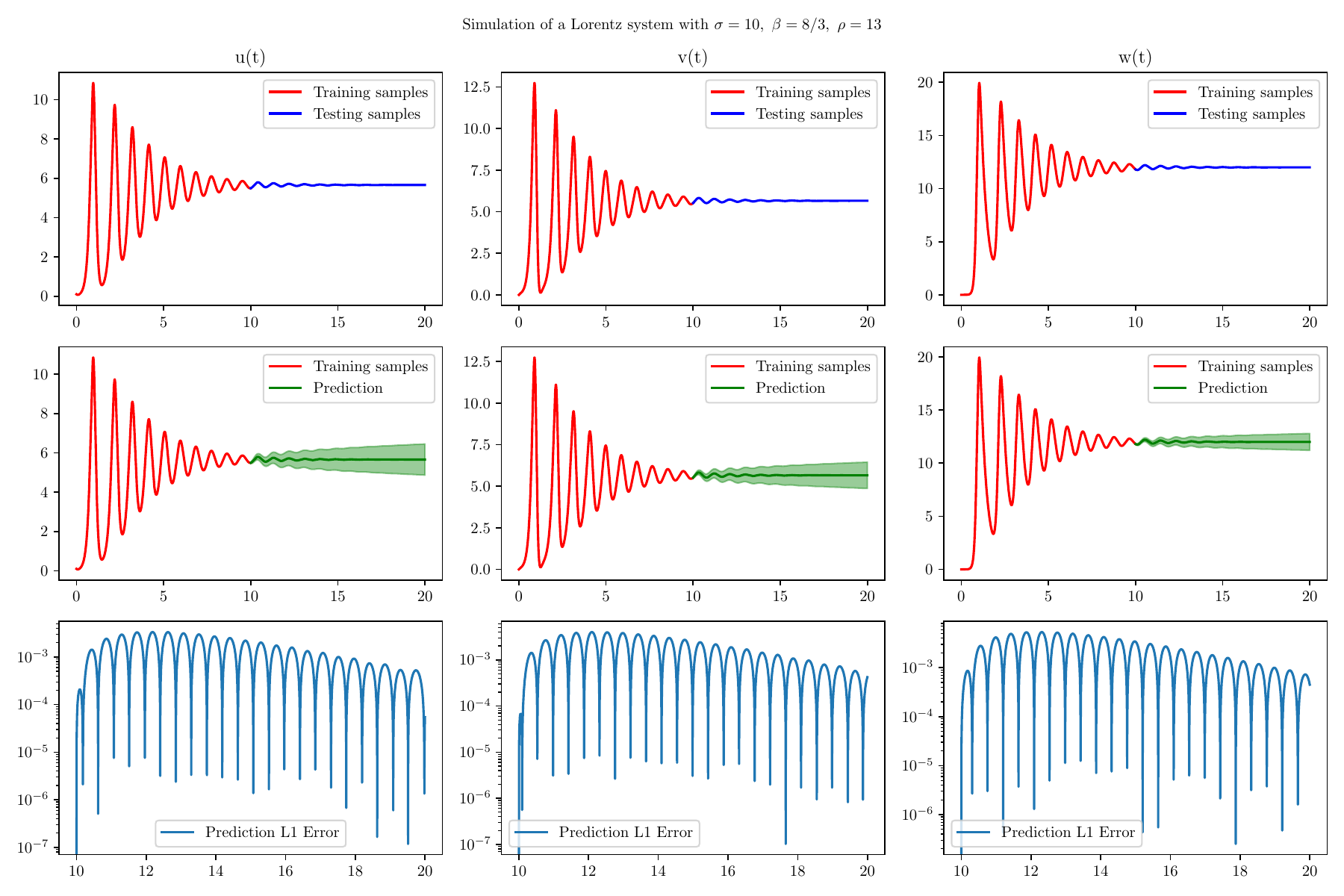}
        \caption{Simulation of a Lorenz system. The error is computed with respect to the mean of the predictions. The bands represent the 95\% confidence interval.}

\label{fig:1}    
\end{figure}

As a sanity check for the QGPRODE, we try to learn the system
\begin{equation}\label{eq:sin}
\left[\begin{array}{c}u(t) \\ v(t) \\ w(t)\end{array}\right]=\left[\begin{array}{ccc}1 & 0.25 & 0.125 \\ -1 & 0.5 & 0 \\ 1 & 0.25 & -0.125\end{array}\right]\left[\begin{array}{c}\sin(t) \\ \cos(t) \\ \sin^2(t)\end{array}\right],
\end{equation}
for time $t \in [0,2\pi]$, which is of interest in our case, as it is quadratic in the latent variable that spans the maximum variable subspace, in alignment with the hypothesis of QGPRODE.
\begin{figure}
\centering
\includegraphics[width=\textwidth]{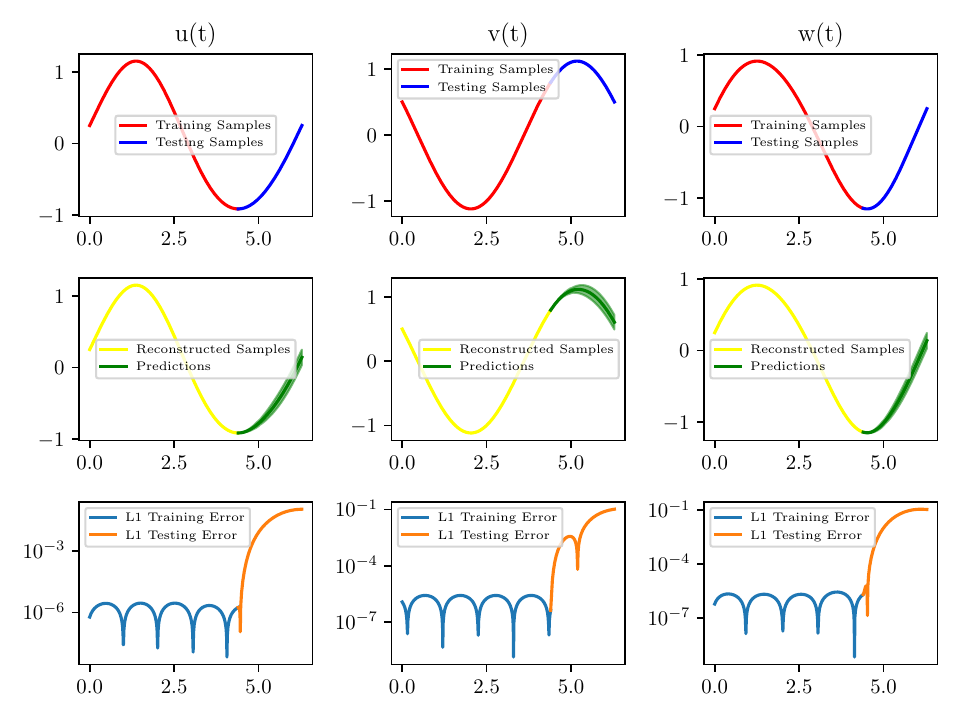}
\caption{Simulation of Eq. \ref{eq:sin}. 
The error is computed with respect to the mean of the predictions. The bands represent the 95\% confidence interval.}
\label{fig:2}    
\end{figure}
We take 1000 time steps and use the first 700 for training. \\
Again, the model is able to learn the trajectory of the system (Fig. \ref{fig:2}).\\
We compare the QGPRODE with LANDO, BOPDMD and EDMD, from the library PyDMD with truncation rank 2 and default parameters. We also measure the performance of GPRODE for completeness. As quadrature weights for QGPRODE, we adopt the Monte Carlo ones, so $q_{i}=\frac{1}{700} \forall i=1...700$. This choice of the quadrature weights will also be adopted for the other test cases.

\begin{table}[h]
\centering

\begin{tabular}{|l|l|l|l|}
\hline
Type     & scaled L1   & relative L2   & Time (seconds) \\ \hline
EDMD     & 0.40 & 1.77 & \textbf{0.2} \\ \hline
BOPDMD   & 0.04 & 0.16 & {0.5} \\ \hline
QGPRODE   & \textbf{0.02}  & \textbf{0.08} & 1.5 \\ \hline
GPRODE   & {0.07}  & {0.29} & 0.3 \\ \hline
LANDO    & 0.42 & 1.77 & 0.68 \\ \hline
\end{tabular}
\caption{Sin-cos-sin$^2$ system. Scaled L1 and relative L2 error between each predictor and the true value.}
\label{tab:2}
\end{table}

As shown in Table \ref{tab:2}, QGPRODE is able to outperform existing methodologies in precision. It is slower because the number of samples is small, and so time integration dominates. As the number of samples is low, both LANDO and EDMD overfit the training data, and LANDO is slower due to the data subsampling. EDMD is the fastest method as it requires neither time integration nor subsampling. However, because it is trained on a relatively small number of samples, it tends to overfit the training data, which in turn leads to a high test error.  BOPDMD, being a linear predictor, is less susceptible to overfitting than EDMD. Nevertheless, it is slower than EDMD because it is an ensemble method and, since the underlying system is near a linear system but not completely linear, it yields a higher error than QGPRODE, but still better than LANDO and EDMD. GPRODE has better results than LANDO and EDMD thanks to the advanced time integration.\\
We have verified that the methodology correctly infers systems that satisfy the hypothesis needed for convergence. We now proceed to investigate the performance of systems in which these assumptions are not fulfilled.\\

\subsection{BV$-\alpha$ model}
First, we try to learn the dynamics of the stream $\psi$ of the variant of the BV-$\alpha$ model introduced in \cite{girfoglio_novel_2023}:
\begin{equation}
\begin{cases}    
\partial_t q+\nabla \cdot((\nabla \times \psi) q)-\frac{1}{\operatorname{Re}} \Delta q=F & \text { in } \Omega \times\left(0, 100\right), \\
-\alpha^2 \nabla \cdot(a(q) \nabla \bar{q})+\bar{q}=q & \text { in } \Omega \times\left(0, 100\right), \\
-\operatorname{Ro} \Delta \psi+y=\bar{q} & \text { in } \Omega \times\left(0, 100\right),
\end{cases}
\end{equation}

 where $q$ is the vorticity, $\bar{q}$ is the filtered vorticity, $a$ is a regularisation function. We refer to \cite{girfoglio_novel_2023} for details of the initial and boundary conditions. 

We test the QGPRODE and the GPRODE methodologies, and we compare them with LANDO, EDMD, and BOPDMD.
Results are summarised in Tab. \ref{tab:3}.

\begin{table}[h]
\centering
\begin{tabular}{|l|l|l|l|}
\hline
Type     & scaled L1   & relative L2   & Time (seconds) \\ \hline
EDMD     & 0.04 & 0.23 & \textbf{0.04} \\ \hline
BOPDMD   & 0.03 & 0.24 & 0.24 \\ \hline
GPRODE   & 0.06 & 0.37 & 0.12 \\ \hline
LANDO    & 0.03 & 0.24 & 0.24 \\ \hline
QGPRODE & \textbf{0.02} & \textbf{0.13} & 0.79 \\ \hline
\end{tabular}
\caption{BV-$\alpha$ test case. Scaled L1 and relative L2 error between each predictor and the true value.}
\label{tab:3}
\end{table}

The temporal snapshots are saved each second, while the simulation is performed with $\Delta t=2.5 \cdot 10^{-4}$. 
The spatial discretisation is $16\times 32$; the first 80 samples are used for training and the last $20$ for testing. Derivatives are computed using a second-order centred difference scheme in the interior and a first-order difference scheme in the boundaries.
EDMD, BOPDMD, LANDO and QGPRODE are truncated to the first 12 singular values.
As shown in Table \ref{tab:3}, QGPRODE achieves the lowest scaled L1 and relative L2 errors, while also maintaining short training and testing times (under 1 second). This is due to its projection onto a sphere and the use of a time integrator that enforces stability, even though the time integration itself makes it the slowest method computationally.  
GPRODE exhibits the largest error because of overfitting (no truncation is applied), whereas BOPDMD is less susceptible to overfitting due to its linear nature. EDMD and LANDO outperform GPRODE due to truncation, but they still overfit and ultimately perform worse than BOPDMD. Both LANDO and BOPDMD are slower than EDMD. In particular, in the case of LANDO, due to the subsampling step and in the case of BOPDMD, due to its ensembling procedure.
In Figs. \ref{fig:3} and \ref{fig:4} the QGPRODE dynamic is compared with that of the true model. We remark that in Fig. \ref{fig:3}, the bound is higher in which there is greater variance in the training data, and that the true variable is between the lower and upper bounds for every space point. The QGPRODE solution has less variability than the true solution due to the sphere projection, which is inherently dissipative in nature.
In Fig. \ref{fig:4}, we observe that the QGPRODE is able to get the distribution of the stream over the space. \\
In Fig. \ref{fig:3_1}, the temporal evolution of the L1 error is shown.

GPRODE does not perform well because, differently from the other methods tested, there is no rank truncation, and the model overfits due to the model misspecification and low data availability.\\

\begin{figure}
\centering
\includegraphics[width=\textwidth]{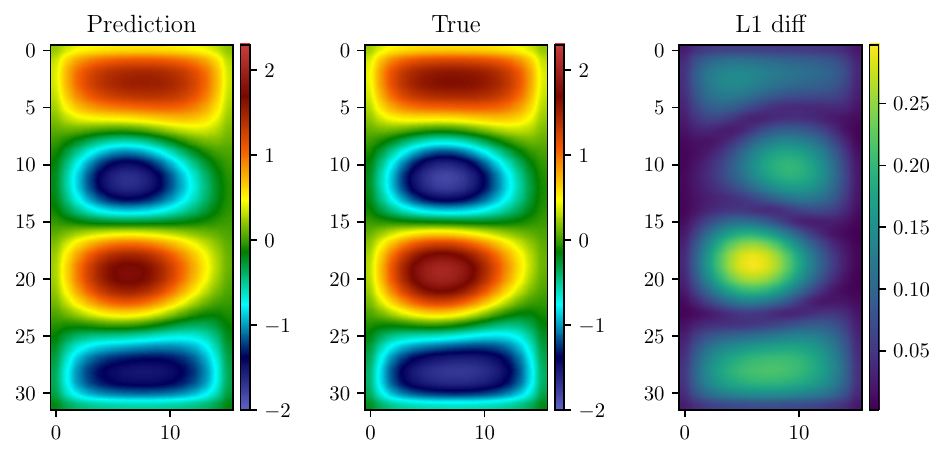}
\caption{BV-$\alpha$ test case. Prediction at time 100 of the QGPRODE with 95\% confidence bounds and the true value. }
\label{fig:3}    
\end{figure}

\begin{figure}
\centering
\includegraphics[width=\textwidth]{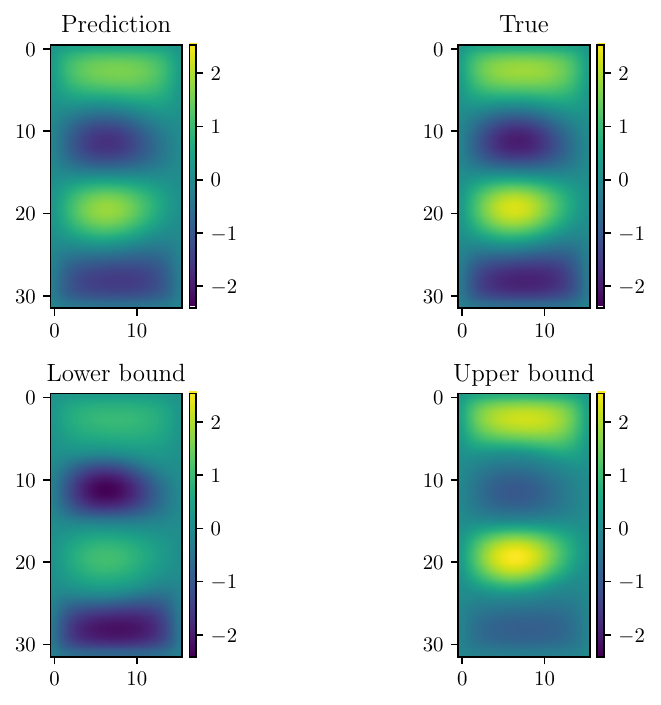}
\caption{BV-$\alpha$ test case. Dynamics of the QGPRODE and of the true model, and the L1 error, averaged through time, with uncertainty bands. }
\label{fig:4}    
\end{figure}

\begin{figure}
\centering
\includegraphics[width=\textwidth]{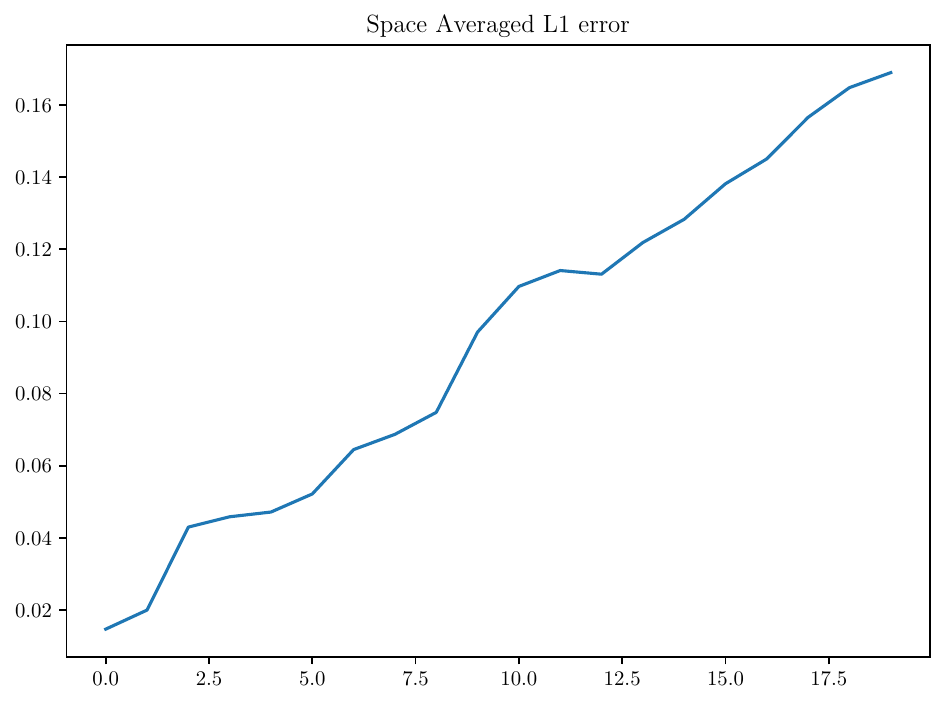}
\caption{BV-$\alpha$ test case. Space-Averaged Absolute Value Difference between the QGPRODE dynamics and the True dynamics. }
\label{fig:3_1}    
\end{figure}

\subsection{Earth air temperature}
As a final test case, we try to predict the Earth's air temperature at different pressure levels in the period between 01-10-2025 and 31-12-2025 over the entire Earth. Training data is taken from the ERA5 dataset \cite{c3s_era5_2018} and is composed of snapshots from 01-01-2025 to 30-09-2025 (1 sample per day). The data is downscaled to 2.5 degrees.

We test the QGPRODE and the GPRODE methodology, and we compare it with LANDO, EDMD, and BOPDMD.\\
Results are summarised in Tab. \ref{tab:4}.

\begin{table}[h]
\centering
\begin{tabular}{|l|l|l|l|}
\hline
Type     & scaled L1   & relative L2   & Time (seconds) \\ \hline
EDMD     & 3.16 & 3.17 & 8 \\ \hline
BOPDMD   & 0.06 & 0.9 & \textbf{6} \\ \hline
GPRODE   & 0.7 & 0.4 & 23.40 \\ \hline
LANDO    & 1.83 & 1.84 & 306 \\ \hline
QGPRODE & \textbf{0.05} & \textbf{0.08} & 9 \\ \hline
\end{tabular}
\caption{ERA5 temperature test case. Scaled L1 and relative L2 error between each predictor and the true value.}
\label{tab:4}
\end{table}

The spatial discretisation is $32 (height)\times 140 (latitude) \times 103 (longitude)$. Derivatives are computed using a second-order difference scheme in the interior and a first-order difference scheme in the boundaries.

QGPRODE requires a few seconds for training and testing. Thus, the proposed methodology accommodates the daily ingestion of newly available data for the generation of Digital Shadows. This feature is particularly advantageous, given that ERA5 is provided with a latency of approximately five days and does not supply forecast fields.\\
EDMD, BOPDMD, LANDO and QGPRODE are truncated to the first 20 singular values.\\

As reported in Table \ref{tab:4}, QGPRODE outperforms existing methodologies in terms of accuracy and remains competitive with respect to computational time. In contrast to EDMD and GPRODE, LANDO require computing the kernel matrix at every time step, which is computationally demanding in high-dimensional state spaces. The runtime of LANDO is the slowest among these methods due to data subsampling, whereas GPRODE is slower than EDMD as a consequence of the additional time-integration step. 
EDMD, GPRODE, and LANDO exhibit overfitting, although GPRODE does so to a lesser extent because of the stabilising effect of time integration. BOPDMD achieves superior performance because it is a linear method that does not overfit and is computationally more efficient than the other algorithms, as it does not require kernel matrix evaluations and involves fewer large-scale matrix operations than QGPRODE.
In Figs. \ref{fig:5} and \ref{fig:6}, the dynamics generated by QGPRODE are compared with the real dynamics. In Fig. \ref{fig:5}, the predictive uncertainty bound is larger in regions where the variance of the training data is higher, and the true state variable remains within the lower and upper bounds in every spatial location. Visual comparisons are restricted to the surface layer of the simulations. 
Furthermore, in Fig. \ref{fig:5}, QGPRODE accurately reproduces the spatial distribution of the stream. GPRODE does not perform well, as, unlike the other methods tested, there is no rank truncation, and the model overfits due to the model misspecification and low data.\\
The upper and lower bounds are high because the variance is increasing in time, and the time interval is very large. In Fig. \ref{fig:6_1}, the temporal evolution of the L1 error is shown.

\begin{figure}
\centering
\includegraphics[width=\textwidth]{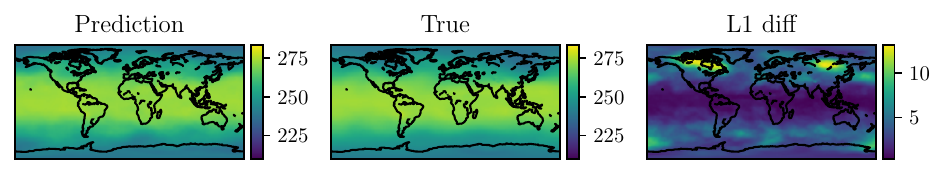}
\caption{ERA5 temperature test case. Prediction at time step 99 of the QGPRODE with confidence bounds and the true value. }
\label{fig:5}    
\end{figure}

\begin{figure}
\centering
\includegraphics[width=\textwidth]{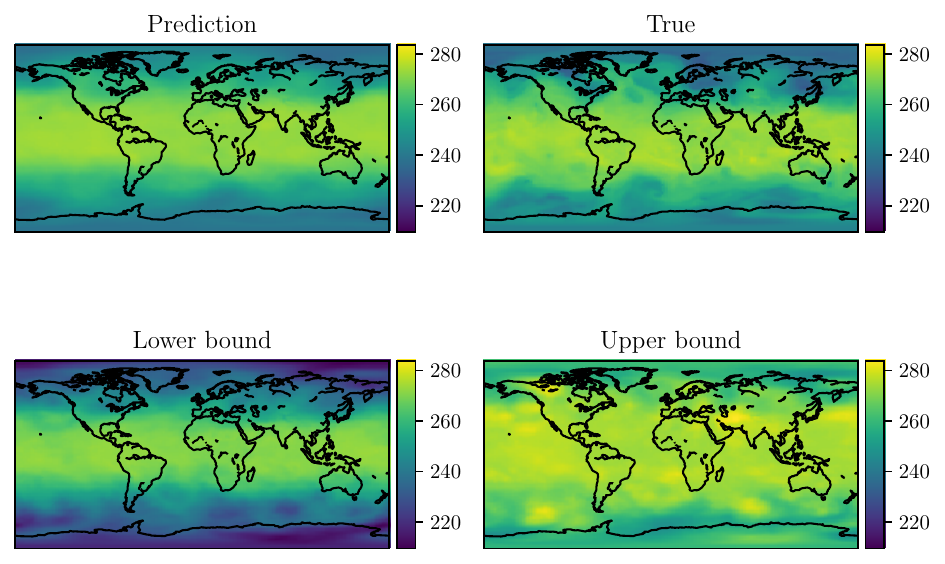}
\caption{ERA5 temperature test case. Dynamics of the QGPRODE and of the true model, and the L1 error, averaged through time, with uncertainty bands.}
\label{fig:6}    
\end{figure}

\begin{figure}
\centering
\includegraphics[width=\textwidth]{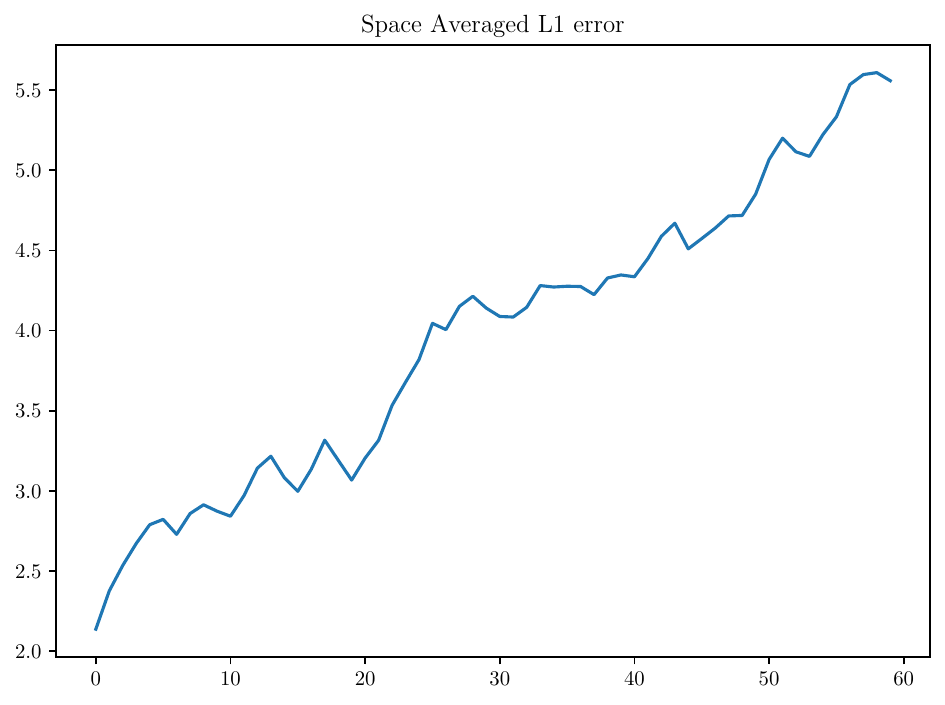}
\caption{ERA5 temperature test case. Space-Averaged Absolute Value Difference between the QGPRODE dynamics and the True dynamics.  }
\label{fig:6_1}    
\end{figure}

\section{Conclusion}

In this work, we have introduced a novel framework for constructing Digital Shadows of complex dynamical systems by leveraging a combination of Gaussian Processes and Stochastic Differential Equations. The proposed GPRODE model integrates kernel-based regression with derivative approximation to infer continuous-time dynamics, while the QGPRODE extension further enforces stability through a quadratic manifold projection onto the unit sphere.

We have rigorously established that, under mild regularity and sampling assumptions, GPRODE converges in the infinite-data limit. This convergence result also extends to certain classes of Stochastic Differential Equations and Random Differential Equations.

Through a series of numerical experiments—including the classical Lorenz system, a synthetic quadratic latent system, the BV-$\alpha$ geophysical flow model, and a real-world near-surface air temperature forecasting problem at the global scale—we have demonstrated that GPRODE and QGPRODE attain competitive short-term prediction accuracy with principled uncertainty quantification, outperforming classical reduced-order modelling (ROM) techniques. Both models exhibit high computational efficiency, which facilitates rapid retraining upon the arrival of new data, a critical requirement for operational Digital Shadow implementations.

Future research will focus on extending this methodology to multiple trajectories by incorporating reduced-order models in the parameter space. Moreover, generalising the framework to controlled dynamical systems and real-time decision support constitutes a promising avenue for practical deployment in engineering applications and environmental monitoring.

\section*{Acknowledgements}
The authors acknowledge the support provided by the European Union NextGenerationEU, in the framework of the iNEST- Interconnected Nord-Est Innovation Ecosystem (iNEST ECS00000043 - CUP G93C22000610007) consortium. The authors acknowledge the financial support provided by the INdAM-GNCS.

\bibliographystyle{unsrt}
\bibliography{MichelePaper}

\begin{comment}
\end{comment}

\appendix
\section{Theoretical Analysis of GPRODE}
In this section, we analyse the convergence properties of the GPRODE methodology. The primary objective of GPRODE is to infer the unknown drift function \(f\) that governs a dynamical system, using observed trajectory data. Rather than estimating \(f\) via a parametric model, we employ a nonparametric formulation based on Gaussian processes. \\ The central idea is to interpret the available data \(B_N\) as (possibly noisy) observations of the underlying vector field and to construct a probabilistic surrogate model for the drift. Within this framework, the posterior mean function \(\mu_N\) provides a point estimate of the drift, while the posterior variance function \(\sigma_N\) characterises the epistemic uncertainty associated with this estimate. \\ This probabilistic representation enables both the reconstruction of the system’s dynamics and the quantification and propagation of uncertainty over time. As additional data are incorporated into the model, the posterior distribution becomes increasingly concentrated, and the variance \(\sigma_N\) is expected to converge to zero under suitable regularity conditions, thereby yielding a deterministic approximation of the true dynamical system.
We provide another notation table for the variables in the appendix.

\begin{longtable}{@{} >{\(}l<{\)} p{9cm} @{}}
\caption{Summary of notation.}
\label{tab:notation2} \\
\toprule
\multicolumn{1}{l}{\textbf{Symbol}} & \textbf{Description} \\
\midrule
\endfirsthead
\multicolumn{2}{c}{\tablename~\ref{tab:notation} (continued)} \\
\toprule
\multicolumn{1}{l}{\textbf{Symbol}} & \textbf{Description} \\
\midrule
\endhead
\bottomrule
\endfoot

\multicolumn{2}{@{}l}{\textit{Data and state}} \\[2pt]
X(t) \in \mathbb{R}^{m}      & State vector of the dynamical system at time $t$ \\
X_{\mathrm{train}}            & Training dataset $\{X(t_1),\ldots,X(t_N)\}$ \\
m                             & State space dimension \\
N                             & Number of training samples \\
T_N                           & Set of training times $\{t_1,\ldots,t_N\}$ \\
B_N                           & Set of training states $\{X(t)\mid t\in T_N\}$ \\
B                             & Closure $\overline{\{X(t)\mid t>0\}}$; assumed compact in $\mathbb{R}^m$ \\
\tilde{B}                     & Open trajectory set $\{X(t)\mid t>0\}$ before closure \\
U                             & Bounded, convex, open set with $B\subset U$ \\
h_{B,B_N}                     & Fill distance $\max_{x\in B}\min_{y\in B_N}\|x-y\|_2$ \\[6pt]

\multicolumn{2}{@{}l}{\textit{Dynamics and operators}} \\[2pt]
f:\mathbb{R}^m\to\mathbb{R}^m & True drift function of the autonomous ODE $\dot{X}=f(X)$ \\
L_f                           & Lipschitz constant of $f$ \\
\bar{L}                       & Uniform bound on the Lipschitz constants $L_N$ of $\mu_N$ \\
M                             & Prediction horizon  \\
y(t,x)                        & Flow of the ODE from initial condition $x$ at time $t$ \\
\Psi                          & Finite-difference operator approximating $\dot{X}(t_i)$ \\
\Psi_N                        & Derivative approximator of order $p$; maps $H^{(m+5)/2}(U)\to H^s(U)$ \\
\Phi                          & Numerical time-stepping function (e.g.\ RK45, symplectic) \\
\Delta t                      & Time step size for numerical integration \\
\overline{\Delta t}_N         & Maximum time spacing $\sup_{i\in T_N}(t_i - t_{i-1})$ in the training set \\
\overline{\Delta t}           & Global upper bound on $\overline{\Delta t}_N$; assumed $<2/L_f$ \\[6pt]

\multicolumn{2}{@{}l}{\textit{Gaussian process and kernel}} \\[2pt]
K(x,y)                        & Matérn kernel $k(\|x-y\|_2)=e^{-d}(1+d)$ \\
\mathrm{RKHS}(U)              & Reproducing Kernel Hilbert Space on $U$ induced by $K$; norm-equivalent to $H^{(m+3)/2}(U)$ \\
\mu_N(x)                      & GP posterior mean with exact derivatives: $K(x,B_N)K(B_N,B_N)^{+}f(B_N)$ \\
\tilde{\mu}_N(x)              & GP posterior mean with approximate derivatives: $K(x,B_N)K(B_N,B_N)^{+}\Psi_N(B_N)$ \\
\sigma_N(x)                   & GP posterior standard deviation \\
\hat{\sigma}_N                & $\sup_{x\in B}\sigma_N(x)$; uniform bound on posterior std \\
b_N(x)                        & Lipschitz-corrected diffusion coefficient derived from $\sigma_N(x)$ \\
f_N \sim \mathcal{N}(\mu_N, \sigma_N^2) & GP surrogate for the drift with exact derivatives \\
\tilde{f}_N \sim \mathcal{N}(\tilde{\mu}_N, \sigma_N^2) & GP surrogate for the drift with approximate derivatives \\
C_{\Psi}(f),\,C_{2,\Psi}(f)  & Approximation constants of the derivative approximator $\Psi_N$ \\
\bar{C}_\Psi(f),\,\bar{C}_{2,\Psi}(f) & Approximation constants in terms of $\overline{\Delta t}_N$ (Assumption~4) \\[6pt]

\multicolumn{2}{@{}l}{\textit{GPRODE model quantities}} \\[2pt]
\tilde{E}_N(t)                & Mean trajectory obtained with approximate derivative $\tilde{\mu}_N$ \\
\tilde{V}_N(t)                & Variance of the trajectory\\
\Sigma(t)                     & Approximate covariance matrix $\mathrm{Cov}_\omega(X(t,\omega))$ \\
\bar{V}                       & Threshold (saturation) variance \\
\gamma(x)                     & Saturation function $x/(x+0.001)$ enforcing $V\le\bar{V}$ \\
A(t)                          & Augmented state $[E(t)^\top,\,V(t)]^\top$ \\
\tilde{A}_N(t)                & Augmented state built from $\tilde{E}_N$ and $\tilde{V}_N$ \\
g(A(t))                       & Right-hand side of the augmented ODE system \\
\tilde{g}_N(\tilde{A}_N(t))   & Right-hand side using approximate mean $\tilde{\mu}_N$ and $b_N$ \\
\tilde{A}_{\Phi,k,N}                  & Numerical solution of the augmented ODE at step $k$ via integrator $\Phi$ \\
K_N                           & Number of integration steps $M/\overline{\Delta t}_N \in \mathbb{N}$ \\
D_\Phi(f,hk)                  & Local error functional of the IVP solver $\Phi$ of order $p$ \\[6pt]

\multicolumn{2}{@{}l}{\textit{Function spaces}} \\[2pt]
H^s(U)                        & Sobolev space of order $s$ on $U$ \\
H^s_0(U,\mathbb{R}^m)         & Sobolev space with zero boundary trace \\
D^{(k+2+m)/2}(U,\mathbb{R}^m) & Set of diffeomorphisms in $H^{(k+2+m)/2}(U,U)$ \\
\mathbb{B}                    & Ball of diffeomorphisms $\{x\in H^{(k+2+m)/2}: \|x-\mathrm{id}\|\le\alpha<1\}$ \\
F(y)                          & Composition operator $x\mapsto f(y(x))$; shown Lipschitz on $\mathbb{B}$ \\[6pt]

\multicolumn{2}{@{}l}{\textit{Non-equispaced derivative approximation}} \\[2pt]
g_\epsilon                    & Best RKHS approximation of $f$ within $\epsilon$-tolerance; exists by universality of the Matérn kernel \\
f_\epsilon^{N,\Delta t_N}     & Solution of the constrained RKHS optimisation for non-equispaced data \\
\epsilon                      & Tolerance parameter in the constrained optimisation ($\epsilon>0$ if $f\notin\mathrm{RKHS}(U)$) \\
\beta_j                      & Dual coefficients in the kernel expansion $\sum_j \alpha_j K(X(t_i),X(t_j))$ \\[6pt]

\\[6pt]

\multicolumn{2}{@{}l}{\textit{Theoretical Analysis of QGPRODE}} \\[2pt]
\bar{E}[F]                    & Discrete time-average quadrature approximation of $(t_N-t_1)^{-1}\int F(t)\,dt$ \\
\phi & Function that maps $\tilde{B}$ to $\mathbb{S}^{k}$ \\
E_{\alpha}[f] & Mean of $f$ according to the ergodic measure $\alpha$. 

\end{longtable}

\subsection{Proof of convergence: exact derivative}
The convergence analysis of the GPRODE framework is structured in three main steps. \\ First, we show that the Gaussian process regression converges uniformly to the true drift function $f$ as the number of observations increases. \\Second, we establish that the solutions of the approximate ordinary differential equations converge to the solution of the true dynamical system. \\Finally, we extend the analysis to stochastic differential equations, proving convergence in mean square. \\ This decomposition allows us to isolate the different sources of approximation error and study them independently.\\
First, we present two assumptions required to establish convergence.
\begin{assumption} \label{assum:1}
We assume that the data inflow follows
\begin{equation}
\dot{X}(t)=f(X(t)),
\end{equation}
where $\overline{\{X(t)|t>0\}}= B$ is compact subset of $\mathbb{R}^{m}$
and exists $U$ bounded, convex and open set such that $B\subset U$ and $f\in H_{0}^{{\frac{m+5}{2}}}({U},\mathbb{R}^{m}).$
\end{assumption}
The assumption of compactness holds true in certain real-world applications, such as oceanographic currents, which are governed by conservation laws.
We will denote by $L_{f}$  the Lipschitz constant of $f$.
We define $$T_{N}
=\{t_{1},...,t_{N}\}$$ as the times of the training set, with $t_{i}<t_{i+1}\quad \forall i=1\ldots N$.\\
We define $$B_{N}
=\{X(t) | t \in T_{N}\}$$ as the training set.
Additionally, let
\begin{equation}
h_{B,B_{N}}=\max\limits_{x\in B}\min\limits_{y\in B_{N}}||x-y||_{2}
\end{equation}
the fill distance. Consistently with \cite{lederer_uniform_2021}, we assume that
\begin{assumption} \label{assum:2}
\begin{equation}
 \lim_{n \rightarrow +\infty}h_{B,B_{N}}=0.
\end{equation}
\end{assumption}
The assumptions introduced in this section are classical in the study of nonparametric dynamical systems. \\ The compactness of the trajectory ensures that the system evolves within a bounded subset of the state space, a property that is fundamental for establishing uniform convergence results. This requirement is fulfilled in a wide range of applications in which conservation laws or physical constraints restrict the evolution of the system. \\ The fill distance assumption guarantees that the observations become dense in the domain as the sample size increases, a condition that is crucial for the statistical consistency of the Gaussian process estimator. \\ The regularity assumption imposed on the drift function $f$ ensures that it lies in the reproducing kernel Hilbert space associated with the chosen kernel, thereby enabling the application of kernel-based approximation theory.\\
We are now prepared to present the methodology in a formal manner and to establish its convergence through rigorous proof.
Our focus is on predicting the behaviour of the ODE within the interval $(t_{N},t_{N}+M]$, where $M$ is independent of $N$. \\
As we did in Section \ref{sec:2}
\begin{equation}
 K(x,y)=k(||x-y||_{2})   
\end{equation}
where
\begin{equation}
    k(d)=e^{-{d}}(1+d)
\end{equation}
which corresponds to the Matérn kernel. This choice of kernel is motivated by the fact that it has been proved in \cite{tuo_improved_2020,wendland_scattered_2004} that
\begin{equation}
RKHS({U})=H^{\frac{m}{2}+\frac{3}{2}}({U})
\end{equation}
with equivalent norms.
It follows that $H_{0}^{{\frac{m+5}{2}}}({U},\mathbb{R}^{m})\subset RKHS({U},\mathbb{R}^{m})$, and thus we can learn the drift function $f$ in this space.\\
Finally, before stating Theorem 1, we define the instantaneous mean function sequences

\begin{equation}
{\mu}_{N}(x)=K(x,B_{N})K(B_{N},B_{N})^{+}f(B_{N}),
\end{equation}

the instantaneous variance function sequence
\begin{equation}
\sigma_{N}(x)=\sqrt{K(x,x)-K(x,B_{N})(K(B_{N},B_{N}))^{+}K(B_{N},x)},
\end{equation}
the modified variance function sequence
\begin{equation}
    b_{N}(x)=
    \begin{cases}
    \sigma_{N}(x) & \sigma_{N}(x) \ge 1\\
    \frac{2 \sigma_{N}(x)^{2}}{1+\sigma_{N}(x)^{2}} & \sigma_{N}(x)<1,
\end{cases}   
\end{equation}
and $f_{N}\sim \mathcal{N}(\mu_{N},\sigma_{N}^{2}).$
The following Theorem can now be stated.
\begin{theorem} \label{thm:1}
It holds
\begin{enumerate}
    \item \begin{equation}
    \begin{gathered} 
b_{N}^{2}(x)\le \sigma_{N}^{2}(x) \\ \leq 1-k(\operatorname{dist}(x,B_{N}))\le \frac{\operatorname{dist}(x,B_{N})}{2} \quad \forall x \in U.
    \end{gathered}
\end{equation}
\item If assumption \ref{assum:2} holds, then 
if we define $\hat{\sigma}_{N}=\sup_{x\in B}\sigma_{N}$, we get
\begin{equation}
\lim_{N\rightarrow +\infty} \hat{\sigma}_{N}= 0\end{equation} uniformly.\\
\item If assumption \ref{assum:1} holds, then
\begin{equation}
\sup_{x\in B}||\mu_{N}(x)-f(x)||_{2}\le 2||f||_{{RKHS}(U,\mathbb{R}^{m})}\sigma_{N}(x), \quad \forall x\in U.
\end{equation}
\item If assumptions \ref{assum:1},\ref{assum:2} hold, then $\sup_{x\in B} ||\mu_{N}(x)-f(x)||_{2}\rightarrow 0$.
\end{enumerate}

\begin{proof}
\begin{enumerate}
    \item The inequality is found in \cite{lederer_uniform_2021,koepernik_consistency_2021} if the inverse matrix is adopted of the pseudoinverse, but the proof follows straightforwardly also for the pseudo-inverse case.
    \item Follows from the previous point.
    \item The scalar version with the inverse matrix instead of the pseudoinverse is found in \cite{maddalena_deterministic_2021,marchi_kernel_2011}, but the proof follows straightforwardly also for the pseudo-inverse case. Our case follows by applying the scalar version to every component and by computing the $\ell_{2}$ norm.
    \item Follows from assumptions and from points 2 and 3.
\end{enumerate}
\end{proof}
\end{theorem}

Theorem \ref{thm:1} establishes the core approximation guarantees for the Gaussian process model. \\The first part shows that the posterior variance $\sigma_N(x)$ can be bounded in terms of the distance to the training points, emphasising how the fill distance governs the level of uncertainty. \\The second part guarantees that the maximum variance converges uniformly to zero, indicating that the model becomes progressively more confident across the entire domain.\\
The third and fourth parts show that the approximation error of the mean function $\mu_N$ is controlled by the variance and therefore also vanishes in the limit.\\
In summary, the theorem formalises the intuition that dense coverage of the domain by data points simultaneously yields lower uncertainty and higher accuracy.\\
In order to establish results about the convergence of ODEs, we define the systems
\begin{equation}
\label{eq:odeexactappex}
\begin{cases}
\dot{{E}}_{N}(t)={\mu}_{N}({E}_{N}(t)) & t\in (t_{N},t_{N}+M],\\ 
\dot{{V}}_{N}(t)=\sigma_{N}({E}_{N}(t))^{2} & t\in (t_{N},t_{N}+M],\\ 
{E}_{N}(t_{N})=Y(t_{N}),\\
{{V}}_{N}(t_{N})=0,\\
\end{cases}
\end{equation}

\begin{equation}
\label{eq:odebappex}
\begin{cases}
\dot{{E}}_{N}(t)={\mu}_{N}({E}_{N}(t)) & t\in (t_{N},t_{N}+M],\\ 
\dot{{V}}_{N}(t)=b_{N}({E}_{N}(t))^{2} & t\in (t_{N},t_{N}+M],\\ 
{E}_{N}(t_{N})=Y(t_{N}),\\
{{V}}_{N}(t_{N})=0,\\
\end{cases}
\end{equation}

\begin{equation}
\label{eq:odeblimappex}
\begin{cases}
\dot{{E}}_{N}(t)={\mu}_{N}({E}_{N}(t)) & t\in (t_{N},t_{N}+M],\\ 
\dot{{V}}_{N}(t)=b_{N}({E}_{N}(t))^{2}\gamma\left(\frac{\bar{V}_{N}-V_{N}(t)}{\bar{V}_{N}}\right) & t\in (t_{N},t_{N}+M] 
{E}_{N}(t_{N})=Y(t_{N}),\\
{{V}}_{N}(t_{N})=0,\\
\end{cases}
\end{equation}

and the SDE

\begin{equation}   
\label{eq:sdeappex}
\begin{cases} 
d{Y}_{N}=\mu({Y}_{N}(t))dt+b({Y}_{N}(t))I_{m}dW & t\in (t_{N},t_{N}+M],\\
Y_{N}(t_{N})={Y}(t_{N}).
\end{cases}
\end{equation}

The dynamical systems introduced in Equations \eqref{eq:odeexactappex}--\eqref{eq:sdeappex} offer distinct mechanisms for propagating the learned dynamics. \\The deterministic ordinary differential equation driven by $\mu_N$ constitutes a direct approximation of the underlying true dynamical system. \\
The variance dynamics characterise the evolution of uncertainty along the trajectory. The modified variance term $b_N$ is introduced to enhance numerical stability, in particular in regimes where the variance attains small values. \\
The stochastic differential equation formulation explicitly incorporates uncertainty via a diffusion term, thereby enabling a probabilistic description of the system’s temporal evolution. \\
Collectively, these formulations yield complementary viewpoints: a deterministic approximation of the mean dynamics and an uncertainty-aware model capturing stochastic variability. 
We now proceed to show that these equations converge to the true underlying ODE. To establish this result, we first present several general theorems concerning the convergence of normal random variables, ordinary differential equations (ODEs), and stochastic differential equations (SDEs).
\begin{proposition}\label{prop:1}
Let $g_{N}$ a sequence of random variables with $E[g_{N}]=\mu_{N}$ and $Var[g_{N}]=\sigma_{N}^{2}$
such that 
\begin{equation}
\mu_{N}\rightarrow \mu
\end{equation}
and 
\begin{equation}
\sigma_{N}^{2}\rightarrow 0.
\end{equation}
Then 
\begin{equation}
E[|\mu-g_{n}|]\rightarrow 0.
\end{equation}
\begin{proof}
We have $E[|g_{n}-\mu|]\le E[|g_{n}-\mu_{n}|]$ + $|\mu-\mu_{n}|\le E [|g_{n}-\mu_{n}|]$ + $|\mu-\mu_{n}|$ $\le \sigma_{n} + |\mu-\mu_{n}|$ due to Cauchy Schwartz, which converges to 0.

\end{proof}
\end{proposition}
For the convergence of ODEs and SDEs, we need stronger assumptions.
\begin{proposition}\label{prop:2}
 Let $f_{n}:U\rightarrow \mathbb{R}^{m}$ a sequence of Lipschitz  functions (with Lipschitz constant $L_{N}$ on U) that converges uniformly to $f$ on a subset $B\subset U$, with $f$ Lipschitz and Lipschitz constant $L$, and with $L_{N}\le \bar{L}$ and $y_{n}$ a sequence that converges to $y$.

Then, let $Y_{n}(t)$ be the solution of the ODE
\begin{equation}
\begin{cases}
\dot{Y}_{N}=f_{N}(Y_{N}) & t\in [t_{N},t_{N}+M],\\
{Y}_{N}(t_{N})=y_{N},
\end{cases}
\end{equation}
and $Y$ is the solution of
\begin{equation}
\begin{cases}
\dot{Y}=f(Y) & t\in [0,M],\\
{Y}(t_{N})=y_{N}.
\end{cases}
\end{equation}
If $Y(t)\subset B$,
we have $$\lim_{N\rightarrow\infty}\sup_{t\in [0,M]}||
Y_{N}(t+t_{N})-Y(t+t_{N})||_{2}=0.$$
\begin{proof}
We have
\begin{equation}
({Y}_{N}(t+t_{N})-{Y}(t+t_{N}))=\int_{t_{N}}^{t+t_{N}}(f_{N}({Y}_{N}(s))-f(Y(s))) ds
\end{equation}
so
\begin{equation}
||({Y}_{N}(t+t_{N})-{Y}(t+t_{N}))||_{2}\le \int_{t_{N}}^{t++t_{N}}||(f_{N}({Y}_{N}(s))-f(Y(s)))||_{2} ds
\end{equation}
\begin{equation}
\le \int_{t_{N}}^{t+t_{N}}||(f_{N}({Y}_{N}(s))-f_{N}(Y(s))||_{2}+||f_{N}(Y(s))-f(Y(s))||_{2} ds
\end{equation}
\begin{equation}
\le \bar{L}\int_{t_{N}}^{t+t_{N}}||({Y}_{N}(s))-(Y(s))||_{2}ds+t\sup_{s\in [0,t]}||f_{N}({Y}(s))-f(Y(s))||_{2}.
\end{equation}
Let $\chi_{N}=\sup_{s\in [t_{N},t_{N}+t]}||f_{N}({Y}(s))-f(Y(s))||_{2}$. Due to the Grönwall inequality, it holds that
\begin{equation}
||({Y}_{N}(t+t_{n})-{Y}(t+t_{n}))||_{2}\le t\chi_{n}e^{Lt},
\end{equation}
which converges to zero as $\chi_{n}$ goes to 0 due to the uniform convergence on $B$.\\
\end{proof}
\end{proposition}

\begin{proposition}
\label{prop:var}
Let $Y(t):[0,M]\rightarrow  \mathbb{R}^{m}$ a time series, $\sigma_{N}:U\rightarrow \mathbb{R}^{+}$ a function that satisfies 
$\sigma_{N}(x) \le C\operatorname{dist}(x,X_{N})$, $Y_{N}(t)$ such that $||Y(t+t_{N}))-Y_{N}(t+t_{N})||_{2}\le C(t)g_{N}$ with $C(t)$ continuos and such that $h_{N}=\operatorname{dist}(Y(t),X_{N})$ and $g_{N}$ go to 0. Let also $\alpha$ non-negative function on $\mathbb{R}^{+}$ such that
$$G(y)=\int_{0}^{y}\frac{y}{\alpha(y)},$$ is continuos, monotone, well defined, satisfies $G(y)y>0$ and $G(0)=0$.

Then if $V_{N}(t)$ is the solution of the ODE
$$
\begin{cases}
\dot{V}_{N}(t)=\sigma_{N}^{2}(Y_{N}(t))\alpha(V_{N}(t)) \quad t \in [t_{N},t_{N}+M],\\
V_{N}(t_{N})=0,
\end{cases}
$$
it holds that  $V_{N}(t+t_{N})$ converges to $0$ as $N$ goes to $+\infty$.
\begin{proof}
We have
$$V_{N}(t+t_{N})=G^{-1}(\int_{t_{N}}^{t+t_{N}}\sigma_{N}(Y_{N}(s)))ds)\le G^{-1}(\int_{t_{N}}^{t+t_{N}}C(s)\operatorname{dist}(Y_{N}(s),X_{N})ds)\le$$

$$G^{-1}(\int_{t_{N}}^{t+t_{N}}C\operatorname{dist}(Y(s),X_{N})ds+\int_{t_{N}}^{t+t_{N}}C||Y(s)-Y_{N}(s)||_{2}ds)$$
which converges to 0 due to $G^{-1}$ being continuos as $G$ is continuous and monotone. 
\end{proof}
\end{proposition}

\begin{proposition}\label{prop:3}
Let $f_{N}$, $Y(t)$, $f$, $\sigma_{N}$ as in propositions \ref{prop:2}, \ref{prop:var}. Furthermore, assume that $\sigma_{N}$ Lipschitz. \\
Then if $Y_{n}(t)$ is the solution of the Ito SDE
\begin{equation}
\begin{cases}
d{Y}_{N}=f_{N}(Y_{N})dt+\sigma_{N}(Y_{N})IdW & t\in [t_{N},t_{N}+M],\\
{Y}_{N}(t_{N})=y_{N},
\end{cases}
\end{equation}

then
$$\lim_{N\rightarrow \infty} \sup_{t\in [0,M]} \mathbb{E}[||Y_{N}(t+t_{N})-Y(t+t_{N})||_{2}^{2}]=0.$$

\begin{proof}
Let $\chi_{n}=\sup_{x\in B}||f(x)-f_{n}(x)||_{2}^{2} $ and let $\hat{\sigma}_{n}=\sup_{x\in B}||\sigma_{n}||_{2}^{2}$ and

The SDE in consideration is equivalent to 

$$Y_{n}(t+t_{N})-Y(t+t_{N})=\int_{t_{N}}^{t+t_{N}}(f(Y_{n}(s))-f(Y(s)))ds+\int_{t_{N}}^{t+t_{N}}\sigma_{N}(Y_{n}(s))dW(s).$$
By applying $||\cdot||_{2}$ and computing the square we get

$$||Y_{N}(t+t_{N})-Y(t+t_{N})||_{2}^{2}\le 2\int_{t_{N}}^{t+t_{N}}||(f(Y_{N}(s))-f(Y(s)))||_{2}^{2}ds+2m\left(\int_{t_{N}}^{t+t_{N}}\sigma_{N}(Y_{N}(s))dW(s)\right)^2.$$
By taking the expectation and applying the Ito isometry, we get
$$\mathbb{E}[||Y_{N}(t+t_{N})-Y(t+t_{N})||_{2}^{2}]\le 2\int_{t_{N}}^{t+t_{N}}\mathbb{E}[||(f_{N}(Y_{N}(s))-f(Y(s)))||_{2}^{2}]+2m\mathbb{E}[\sigma_{N}^{
2
}(Y_{N}(s))]ds.$$

To bound the right end side, note that
$$||(f_{N}(Y_{N}(s))-f(Y(s)))||_{2}^{2} \le 2\bar{L}^{2} ||Y_{N}(s)-Y(s)||_{2}^{2}+2\chi_{n}^{2},$$

and $$\sigma_{N}^{2}(Y_{N}(s))\le 2C^{2}h_{N}^{2}+2C^{2}||Y_{N}(s)-Y(s)||_{2}^{2},$$ so we get

$$\mathbb{E}[||Y_{N}(t+t_{N})-Y(t+t_{N})||_{2}^{2}]\le 
(4\bar{L}^{2}+4mC^{2})\int_{t_{N}}^{t}\mathbb{E}[||Y_{N}(s)-Y({s})||_{2}^{2}]ds+4mtCh_{N}^{2}+4t\chi_{n}^{2}.$$
By Grönwall
$$\mathbb{E}[||Y_{N}(t+t_{N})-Y(t+t_{N})||_{2}^{2}]\le (4mtCh_{n}+4t\chi_{n})e^{(4\bar{L}+4mC)t}$$
which converges to 0.

\end{proof}
    
\end{proposition}

We are now ready to prove the main theorem of this section.

\begin{theorem}
\label{thm:2}
If  assumptions \ref{assum:1}, \ref{assum:2}, hold, then
\begin{enumerate}

\item \begin{equation}
\lim_{N \rightarrow +\infty}  E[|f_{N}-f|]=0.
\end{equation}

\item Let $E_{N}(t)$, $V_{N}$ the solutions of Eq. \ref{eq:odebappex}, or \ref{eq:odeblimappex}, \ref{eq:odeexactappex}.
Then given the sequence of stochastic process $Y_{N}(t+t_{N})$  with mean ${E}_{N}(t)-X(t+t_{N})$ and variance ${V}_{N}(t)$ if holds that the quantity
$\sup_{t\in [0,M]}\mathbb{E}[||Y_{N}(t+t_{N})-Y(t+t_{N})||_{2}]$
converges when $N\rightarrow +\infty$ to $0$, $\forall t \in [0,M]$.

\item Let $Y_{N}$ the solution of Eq. \ref{eq:sdeappex}. Then  $E[||Y(t+t_{N})-Y_{N}(t+t_{N})||_{2}]$ converges in mean square to  $0$ $\forall t \in [0,M]$.

\end{enumerate}
\end{theorem}
\begin{proof}
    \begin{enumerate}
    \item Follows from Theorem \ref{thm:1} and Proposition \ref{prop:1}.
    \item Follows from Theorem \ref{thm:1}  and Propositions \ref{prop:2},\ref{prop:var}.
    \item Follows from Theorem \ref{thm:1} and Propositions 
     \ref{prop:3}.
        
    \end{enumerate}
\end{proof}

We have demonstrated that the GPRODE framework yields a mathematically consistent procedure for inferring dynamical systems from observational data. \\
In particular, the Gaussian process–based estimator converges uniformly to the true drift function, and the corresponding ODE and SDE models converge to the underlying system dynamics. \\
Collectively, these results furnish a rigorous theoretical foundation for the application of GPRODE in settings that demand both accurate prediction and principled uncertainty quantification.

\subsection{Proof of convergence: numerical approximation}
In practical applications, the time derivative $\dot{X}(t)$ is typically not directly observable and must be approximated from discretely sampled data. Moreover, the resulting system of differential equations must be solved using numerical integration schemes. These two steps introduce distinct sources of numerical error: one associated with the approximation of the derivative and another arising from the time-integration procedure. \\
The analysis presented in this section demonstrates that, under appropriate regularity and sampling assumptions, both error contributions converge to zero as the sampling becomes sufficiently dense. Consequently, the numerical implementation is asymptotically consistent with the underlying theoretical model.\begin{assumption}\label{assum:3}
Let
\begin{equation}
\overline{\Delta t_{N}}=\sup\limits_{i\in T_{N}}{t_{i}-t_{i-1}},
\end{equation}
we assume
\begin{equation}
\sup\limits_{N\in\mathbb{N}}\overline{\Delta t_{N}}=  {\overline{\Delta t}} < \frac{2}{L_{f}}, 
\end{equation}
and
\begin{equation}
\lim_{N\rightarrow +\infty }{\overline{\Delta t_{N}}}=0.
\end{equation}
\end{assumption}

We now define a finite-difference method of order $p$.

\begin{definition}\label{def:fd}
 A family of derivative approximators of order $p\ge 1$ is a sequence of operators
$\Psi_{N}:H^\frac{{m+5}}{2}(U)\rightarrow H^{s}(U)$
with $s=\frac{3+m}{2}\le s\le\frac{{m+5}}{2}$, such that 

\begin{equation}
\sup\limits_{x\in {U}} ||f(x)-\Psi_{N}[f](x)||_{2}\le C_{\Psi}(f) \frac{1}{{N}^{p}}, 
\end{equation}
and
\begin{equation}
||f(x)-\Psi_{N}[f]||_{H^{s}(U)}\le  C_{2,\Psi}(f) \frac{1}{{N}^{p}}\quad  \forall f\in H_0^{{\frac{m+5}{2}}}({U}).
\end{equation}
\end{definition}
Examples include the forward difference operator
\begin{equation}
\sqrt{N}\left({y\left(\frac{1}{\sqrt{N}},x\right)-x}\right),
\end{equation}
where 
\begin{equation}
\begin{cases}
\dot{y}(t ,x)={f}(y(t,x)),\\
y(0,x)=x,
\end{cases}
\end{equation}
which has an order of convergence $0.5$,
and has $C_{1, \Psi,}(f)$ given by $\frac{1}{2}\sup\limits_{x\in {U}}||\nabla_{x}f(x)||_{2} ||f(x)||_{2}$, which is finite to $f$ being $C^{1}$. \\
We prove that it also respects the second requirement in Appendix A with $s=\frac{m+3}{2}$.
We note that $\Psi_{N}$ depends on $f$, which is unknown. We need to further assume that we can infer some information from the data only.
\begin{assumption}\label{assum:4}
We assume that the values of $\Psi_{N}[f](t_{i})$, $i=1...N$ are computable knowing only $T_{N}$ and $B_{N}$ and that

\begin{equation}
\sup\limits_{x\in {U}} ||f(x)-\Psi_{N}[f](x)||_{2}\le \bar{C}_{\Psi}(f) \overline{\Delta t_{N}}^{q},
\end{equation}
and
\begin{equation}
||f(x)-\Psi_{N}[f]||_{H^{s}(U)}\le  \bar{C}_{2,\Psi}(f) \overline{\Delta t_{N}}^{q} \quad  \forall f\in H_0^{{\frac{m+5}{2}}}({U}),
\end{equation}
for a certain $q>0$.
\end{assumption}
The forward difference method satisfies this assumption in the case of equispaced data (the constants are the same as Definition \ref{def:fd}). We need the assumption of $s=\frac{3+m}{2}$ because we are working with a kernel with that regularity, but it can be changed assuming different hypotheses for $f$ or by changing the kernel. We were unable to construct a general difference scheme for non-equispaced data that satisfies the assumption \ref{assum:4}. We propose an alternative later in this Appendix.

Let 
\begin{equation}
\tilde{\mu}_{N}(x)=K(x,B_{N})K(B_{N},B_{N})^{+}\Psi_{N}(B_{N}),
\end{equation}

\begin{equation}
\label{eq:odeappendpsi}
\begin{cases}
\dot{\tilde{E}}_{N}(t)=\tilde{\mu}_{N}(\tilde{E}_{N}(t)) & t\in (t_{N},t_{N}+M],\\ 
\dot{\tilde{V}}_{N}(t)=b_{N}(\tilde{E}_{N}(t))^{2}\gamma\left(\frac{\bar{V}_{N}-\tilde{V}_{N}(t)}{\bar{V}_{N}}\right) & t\in (t_{N},t_{N}+M],\\ 
\tilde{E}_{N}(t_{N})=Y(t_{N}),\\
{\tilde{V}}_{N}(t_{N})=0,\\
\end{cases}
\end{equation}

\begin{equation}
\label{eq:sdeappendpsi}
\begin{cases} 
dY_{N}=\tilde{\mu}({Y}_{N}(t))dt+b({Y}_{N}(t))I_{m}dW & t\in (t_{N},t_{N}+M],\\
{Y}_{N}(t_{N})=Y(t_{N}).
\end{cases}
\end{equation}

We define $\tilde{f}_{N}(x)$ as a sequence of random variables with mean $\tilde{\mu}_{N}(x)$ and variance $\sigma_{n}(X)$.\\
We are ready to state and prove the following Theorem.

\begin{theorem}
If the assumptions \ref{assum:1}, \ref{assum:2},  \ref{assum:3}, \ref{assum:4}  hold, then
\begin{enumerate}
\item     $\tilde{\mu}_{N}\rightarrow f$
uniformly.

\item $ \lim_{N \rightarrow +\infty}  E[|\tilde{f}_{N}-f|]=0$

\item Let 
\begin{equation}
{Y}_{N}(t)\sim MN(\tilde{E}_{N}(t),{V}_{N}(t))
\end{equation}
where $\tilde{E}_{N}(t),\tilde{V}_{N}(t)$ are defined in Eq. \ref{eq:odeappendpsi}. Then $\sup_{t\in [0,M]}\mathbb{E}[||{Y}_{N}(t+t_{N})-Y(t+t_{N})||_{2}^{2}]$
converges to $0$.

\item Let $\tilde{Y}_{N}(t)$ the solution of Eq. \ref{eq:sdeappendpsi}. The $\sup_{t\in [0,M]}\mathbb{E}[||{Y}_{N}(t+t_{N})-Y(t+t_{N})||_{2}^{2}]$
converges to $0$.
\end{enumerate}
\end{theorem}
\begin{proof}
    \begin{enumerate}
        \item We have 
        \begin{equation}   
        \begin{gathered}
        ||(\tilde{\mu}_{N}-f(x))||_{2}\le \\
        ||(\tilde{\mu}_{N}-\Psi_{N}[f](x))||_{2}+||(f(x)-\Psi_{N}[f](x))||_{2}\\
    \le \sup_{x\in B} ||(\tilde{\mu}_{N}-\Psi_{N}(x))||_{2}  +C_{\Psi_{N}}(f)(\overline{\Delta t_{N}})^{q}.
    \end{gathered}
    \end{equation}
    
    As a consequence
    \begin{equation}
        \sup_{x\in B} ||(\tilde{\mu}_{N}-f)|| \le 2 ||\Psi_{N}(x)||_{H^{\frac{m+3}{2}}(U)}\sigma_{N}(x)
        +C_{\Psi_{N}}(f)(\overline{\Delta t_{N}})^{q},
\end{equation} due to the equivalence of norms.
    As $||\Psi_{N}||_{H^{\frac{3+m}{2}}(U)}$ converges uniformly to $||f||_{H^{\frac{3+m}{2}}(U)}$ because of its definition and assumption \ref{assum:4}, everything converges to $0$ because of Theorem \ref{thm:1} and assumption \ref{assum:3}.
    \item Follows from point 1 and Proposition \ref{prop:1}.
    \item Follows from point 1 and Proposition \ref{prop:2}.
    \item Follows from point 1 and \ref{prop:3}.

    \end{enumerate}
\end{proof}

We have shown convergence of the ODE with derivatives estimated using a finite-difference scheme to the true ODE.

\begin{definition}\label{def:ivp}
A functional $\Phi(X,f,h)$ is an Initial Value Problem solver of order $p$ if
given the sequence $X_{\Phi}$, $k\in \mathbb{N}$ given by
\begin{equation}
\begin{cases}
X_{\Phi,0,f}=X,\\
X_{\Phi,k+1,f}=\Phi(X_{\Phi,k,f},f,h),
\end{cases}
\end{equation}

if it holds $\forall k \in \mathbb{N}$
\begin{equation}
||X_{\Phi,k,f}-X(hk)||_{2}\le D_{\Phi}(f,hk)h^{p} \quad \forall f \in H_{0}^{{\frac{5+m}{2}}}({U})
\end{equation}
where $D_{\Phi}(f,hk)$ is a continuous functional and increasing in the second term.
\end{definition}

IVP solvers include both multistep methods (like Adams-Bashforth) and Runge-Kutta methods.

 Let
\begin{equation}
A(t)=\begin{bmatrix}
X(t)\\
0
\end{bmatrix},
\end{equation}

\begin{equation}
\tilde{A}_{N}(t)=\begin{bmatrix}
\tilde{E}_{N}(t)\\
\tilde{V}_{N}(t)
\end{bmatrix},
\end{equation}

\begin{equation}
\tilde{g}_{N}(\tilde{A}_{N}(t))=\begin{bmatrix}
\tilde{\mu}_{N}(\tilde{E}_{N}(t))\\
{b}_{N}^{2}(\tilde{E}_{N}(t))\gamma\left(\frac{\bar{V}_{N}-\tilde{V}_{N}(t)}{\bar{V}_{N}}\right)\\
\end{bmatrix},
\end{equation} we can rewrite the ODE system as
\begin{equation}\label{eq:odediffrec}
\begin{cases}
\tilde{A}_{N}(0)=A(t_{N}),\\
\dot{\tilde{A}}_{N}=\tilde{g}_{N}(\tilde{A}_{N}).
\end{cases}
\end{equation}
For simulating the ODE \ref{eq:odediffrec}, we use the same time stepping that is used for training, so $\overline{\Delta t _ {N}} $.   \\
Then we discretise it as
\begin{equation}
\begin{cases}
A_{\Phi,0,N}=A(t_{N}),\\
A_{\Phi,k,N}=\Phi(A_{\Phi,k-1,N},\tilde{g}_{N},\bar{\Delta t_{N}}).\\
\end{cases}
\end{equation}

\begin{theorem}
If assumptions \ref{assum:1}, \ref{assum:2}, \ref{assum:3}, \ref{assum:4}, hold and furthermore $\frac{M}{\overline{\Delta t_{N}}}=K_{N}\in \mathbb{N} \forall N.$
Then
\begin{itemize}
\item $
\lim\limits_{N\rightarrow +\infty} ||E_{\Phi,K_{N},N}-X(t_{N}+M)||_{2}=0 $
and $\lim\limits_{N\rightarrow +\infty} |V_{\Phi,K_{N},N}-0|=0.
$.
\item Let $X_{\Phi,K_{N},N}\sim MN(E_{\Phi,K_{N},N},V_{\Phi,K_{N},N})$, then
\begin{equation}
E[||X_{\Phi,K_{N},N}-X(M)||_{2}]\rightarrow 0.
\end{equation}

\end{itemize}
\begin{proof}
\begin{itemize}
    \item 
    \begin{equation}
    ||A(M)-A_{\Phi,K_{N},N}||_{2}\le ||A(M)-A_{N}(M)||_{2}+||A_{\Phi,K_{N},N}-A_{N}(M)||_{2}
    \end{equation}
    \begin{equation}
    \le ||A(M)-A_{N}(M)||_{2}+ \bar{D}(\Phi_{N})\overline{\Delta t_{N}}^{p},
    \end{equation}
    which converges to zero because of Theorem 2 and assumption \ref{assum:3}.
    \item Consequence of point 1 and Proposition \ref{prop:1}.
\end{itemize}
\end{proof}
\end{theorem}

\subsection{Proof regarding the forward difference scheme}
In this section, we 
prove that the forward difference scheme is convergent in $H^{\frac{k+m}{2}}(U)$.
Let define $D^{\frac{k+2+m}{2}}({U},\mathbb{R}^{m})$ the set of diffeomorphism which belongs $H^{\frac{k+2+m}{2}}({U},{U})$.
First, we need the following proposition:
\begin{proposition}
Let $\mathbb{B}=\{x\in H^{\frac{k+2+m}{2}}({U},\mathbb{R}^{m}) \text{ s.t.} ||x-\operatorname{id}||\le \alpha <1 \}$, $k\ge 3$ then 

\begin{enumerate}
    \item The operator $F:\mathbb{B} \rightarrow H^{\frac{k+m}{2}}({U},\mathbb{R}^{m})$ given by $F(y)=f(y(x))$ is well defined and furthermore Lipschitz. 
    \item it exists a time $h_{0}$ such that $X(h,x)\in \mathbb{B}, \forall h\le h_{0}.$
\end{enumerate}
\end{proposition}
\begin{proof}
\begin{enumerate}
    \item 
    Let $y \in B$. As $||Dy-\operatorname{I}||_{\infty} \le ||y-\operatorname{id}||<1$ we have that $y$ is invertible and $\mathbb{B}\subset D^{\frac{k+2+m}{2}}({U},\mathbb{R}^{m})$ and a consequence $F$ is well defined \cite{inci_regularity_2013} .
    Now $||F(y)-F(z)||_{H^{\frac{k+m}{2}}}=||\int_{0}^{1}(\nabla f(y(x)+t(y(x)-z(x)))dt)^{T}(y(x)-z(x))||_{H^{\frac{k+m}{2}}}$,
    so as $H^{{\frac{k+m}{2}}}(U)$ is an algebra \cite{behzadan_multiplication_2021} we can bound everything by $$\int_{0}^{1}||(\nabla f(y(x)+t(y(x)-z(x)))dt)||_{H^{\frac{k+m}{2}}} ||(y-z)||_{H^{\frac{k+m}{2}}}.$$
    As $(y(x)+t(y(x)-z(x)))\in \mathbb{B}$,
    everything is bounded by $$||\nabla f||_{H^{\frac{k+m}{2}}}||y-z||_{H^{\frac{3+m}{2}}}.$$
    Thus, $F$ is Lipschitz.
    \item As $F$ is Lipschitz, the thesis follows from the local uniqueness theorem of ODE in Banach Spaces $\cite{pata_fixed_2019}$. 
\end{enumerate}
\end{proof}

We are now ready to prove the convergence of the Forward difference method.
Lets consider
\begin{equation}
||f(x)-\frac{X(h,x)-x}{h} ||_{H^{\frac{3+m}{2}}(U)}.
\end{equation}
By applying the definition of $X(h,x)$ and by changing variables, we arrive at
\begin{equation}
||\int_{0}^{1} \left(f(x)-f(X(hu,x))\right) du ||_{H^{\frac{3+m}{2}}(U)}.
\end{equation}
As $F$ is a Lipschitz operator, we can bound everything by 
\begin{equation}
C\int_{0}^{1}||(X(hu,x)-x) ||_{H^{\frac{3+m}{2}}(U)}du\le C \int_{0}^{1} ||\int_{0}^{hu}f(X(s,x)) ||_{H^{\frac{3+m}{2}}(U)}ds\le CC_{1}h
\end{equation}
as $X(s,x)$ is a diffeomorphism. \\

\subsection{What to do if the training times are not equispaced}
In this section, we study how we can modify our algorithm for accounting for non-equispaced times without assuming Assumption \ref{assum:4}. We focus on the forward difference method for derivative approximation, but it can be adapted to other differentiation algorithms.
As additional requirements, we need a bound to know the constant $M$ such that $\sup\limits_{x\in B} f(x)$, but we can relax the constraint of $f\in RKHS(U)$ to $f$ continuous and Lipschitz with constant $L$. However the computational costs moves from $N^{3}$ to $\log(N)N^{5.67}$.\\
Before stating the modified procedure, we remark that  $\forall \epsilon>0$, the   optimisation problem
\begin{equation}\inf \limits_ {\{g\in RKHS(U) | \sup_{x\in U}||g(x)-f(x)||_{2}\le \epsilon\}} ||g||_{RKHS(U)}
\end{equation}
has at a least a solution $g_{\epsilon}\in RKHS(U)$ due to the universality of the Matérn kernel \cite{micchelli_universal_2006}.
Now let $\Delta t_{N}=\sup\limits_{i\in {1\ldots N-1}} t_{i+1}-t_{i}$,  
As a consequence it holds $||\frac{X(t_{i+1})-X(t_{i})}{t_{i+1}-t_{i}}-f(X(t_{i}))||_{2}\le \frac{1}{2} LM \Delta t_{N}$,
where $M=\max_{x\in U}||f||_{2}$.
Let us consider the optimisation problem
\begin{equation}
\inf \limits_ {\{g\in RKHS(U) | \sup_{i=1...N}||g(X(t_{i}))-\frac{X(t_{i+1})-X(t_{i})}{t_{i+1}-t_{i}}||_{2}\le \epsilon+LM\Delta t_{N}\}} ||g||_{RKHS(U)}
\end{equation}
and its solution $f_{\epsilon}^{N,\Delta t_{N}}$. As $g_{\epsilon}$ is feasible, we have that $||f_{\epsilon}^{N,\Delta t_{N}}||_{RKHS(U)}\le ||g_{\epsilon}||_{RKHS(U)}$ and
\begin{equation}
\begin{gathered}
||f(x)-f_{\epsilon}^{N,\Delta t_{N}}(x)||_{2}\le ||f(x)-f(X(t_{i}))+f(X(t_{i}))-f_{\epsilon}^{N,\Delta t_{N}}(X(t_{i}))\\+ f_{\epsilon}^{N,\Delta t_{N}}(X(t_{i}))- f_{\epsilon}^{N,\Delta t_{N}}(x)||_{2}\le L||x-X(t_{i})||_{2}\\+\epsilon+LM\Delta t_{N}+||g_{\epsilon}||_{RKHS(U)}\sqrt{2-2K(x,X(t_{i}))} \quad \forall i \in 1\ldots N.
\end{gathered}
\end{equation}
Using the definition of fill distance, we arrive at
\begin{equation}
\begin{gathered}
\sup\limits_{x \in B}||f(x)-f_{\epsilon}^{N,\Delta t_{N}}(x)|| _{2}  \\ \le Lh_{B,B_{N}}+\epsilon+LM\Delta t_{N}+||g_{\epsilon}||_{RKHS(U)}\sqrt{2-2k(h_{B,B_{N}})}.
\end{gathered}
\end{equation}
Thus under Assumption \ref{assum:2} and \ref{assum:3} we have that $\lim\limits_{N\rightarrow +\infty}\sup\limits_{x\in B} ||f(x)-f_{\epsilon}^{N,\Delta t_{N}}(x)||_{2}\le \epsilon$.
Thus $f_{\epsilon}^{N,\Delta t_{N}}$ can be adopted instead of $\tilde{\mu}_{N}$ in the GPRODE and the QGPRODE algorithms.\\
We remark that $\epsilon$ can be chosen to be zero only if $f\in RKHS(U)$, otherwise it can be anything strictly bigger than $0$, keeping in mind that $||g_{\epsilon}||_{RKHS(U)}$ is strictly decreasing in $\epsilon$, and $||g_{0}||_{RKHS(U)}=+\infty$ if $f \not \in RKHS(U)$.
It remains to explain how to solve the above optimisation problem. Exploiting the properties of reproducing kernel Hilbert spaces (RKHSs) \cite{maddalena_deterministic_2021}, the problem is equivalent to solving
\begin{equation}
\inf_{\substack{\beta\in \mathbb{R}^{N} \\[2pt] \sup\limits_{i=1,\ldots,N}\left\lVert\sum_{j=1}^{N}\alpha_{j}k\bigl(X(t_{i}),X(t_{j})\bigr)-\dfrac{X(t_{i+1})-X(t_{i})}{t_{i+1}-t_{i}}\right\rVert_{2}\le \epsilon+LM\Delta t_{N}}}
\;\sum_{i,j=1}^{N} \beta_{j}\beta_{i}K\bigl(X(t_{i}),X(t_{j})\bigr),
\end{equation}
which defines a convex quadratic programming problem. This optimisation problem can be solved with a time complexity of $O(\log(N)N^{5.67})$ \cite{kapoor_fast_1986}. Consequently, it is computationally less efficient than the algorithm with time complexity $O(N^{3})$, which, however, requires the use of a differentiation method satisfying Assumption \ref{assum:4}. Consequently, the selection of numerical methods exerts a substantial influence on computational efficiency. \\ For equispaced data, finite-difference schemes enable numerically efficient implementations, exhibiting cubic computational complexity due to the inversion of the associated kernel matrix. \\ In contrast, for non-equispaced data, more general optimisation-based strategies are required, which generally incur higher computational costs. \\ This trade-off underscores the critical role of data structure in determining the practical feasibility and performance of numerical algorithms.

\section{Theoretical analysis of the QGPRODE}

The QGPRODE framework extends GPRODE by incorporating geometric structure in the data. \\ The key idea is that the trajectory of the dynamical system often lies on a low-dimensional manifold embedded in a higher-dimensional space. By mapping this trajectory onto a sphere via a smooth embedding, we can exploit geometric regularity.\\ This transformation allows us to construct reduced-order models that capture the essential dynamics while reducing complexity. \\The use of Taylor expansions and quadratic approximations further enables efficient representation of the system, connecting this approach with model order reduction techniques.

Let $\tilde{B}=\{X(t)|t > 0\}$, so that $B$ is the closure of $\tilde{B}$. 
We require ergodicity, in the sense that we assume that there exists a measure $\alpha:B: \rightarrow \mathbb{R}$ such that
\begin{equation}
\lim\limits_{T \rightarrow +\infty} \frac{1}{T} \int_{0}^{T}l(X(t))dt:=E_{\alpha}[l]=\int_{B}ld\alpha \quad \forall l\in C(B)
\end{equation}
and we furthermore require that $f\in C_{0}^{\lceil{\frac{5+m}{2}}\rceil}(U)$.
In this specific case, as we are considering only one trajectory, ergodicity is equivalent to strong ergodicity.
As a consequence, in the infinite-data limit, the operator $\bar{E}[F(X(\cdot))]$ converges to $E_{\alpha}[F]$.
The assumptions of compactness and ergodicity are also used in proofs of similar techniques, like Extended DMD \cite{korda_convergence_2018}.\\

Note that not all dynamical systems that respect Assumption \ref{assum:1} are ergodic, for example 
\begin{equation}
\begin{cases}
    \dot{x}_{1}=\cos(\operatorname{arctanh}(x_{2}))e^{-\operatorname{arctanh}(x_{2})}, \\
    \dot{x}_{2}=(1-x_2^{2})e^{-\operatorname{arctanh}(x_{2})},
\end{cases}
\end{equation}
which has the solution 
\begin{equation}
\begin{cases}
x_{1}(t)=\sin(\ln(1+t)), \\
x_{2}(t)=\tanh(\ln(1+t)),
\end{cases}
\end{equation}
because
\begin{equation}
\liminf \limits_{T \rightarrow +\infty} \frac{1}{T} \int_{0}^{T}\sin(\ln(1+t))dt=-1,
\end{equation}
and
\begin{equation}
\limsup \limits_{T \rightarrow +\infty} \frac{1}{T} \int_{0}^{T}\sin(\ln(1+t))dt=+1.
\end{equation}

Due to Whitney Theorem, $\tilde{B}$ is $C^{\lceil{\frac{m+5}{2}}\rceil}$ diffeomorphic to a subset of a $k$-dimensional sphere $\mathbb{S}^{k}$ $\forall k\ge 3$, and there exists $\phi\in C^{\lceil{\frac{m+5}{2}}\rceil}(\tilde{B},\mathbb{S}^{k})$ invertible such that $\tilde{B}=\phi^{-1}(\phi(\tilde{B})).$

Let $Z(t)=\phi(X(t))$. It holds that
\begin{equation}
\dot{Z}(t)=\nabla \phi (\phi^{-1}(Z(t)))\cdot f(\phi^{-1}(Z(t))).
\end{equation}

As $g$ is a $C^{\lceil{\frac{5+m}{2}}\rceil}$ diffeomorphism, the right side is Lipschitz and the ODE is well defined. Note that the right side is $C^{\lceil{\frac{3+m}{2}}\rceil}$.
\\
Furthermore, as the right end side is defined on a $C^{\lceil{\frac{3+m}{2}}\rceil}$ dimensional curve in $\mathbb{S}^{k}$, there exists a tubular neighbourhood with projection being $C^{\lceil{\frac{m+3}{2}}\rceil}$ \cite{hirsch_differential_1976} and thus the right end side extends to the tubular neighbourhood as a function $C^{\lceil{\frac{3+m}{2}}\rceil}$.
Then we can extend the right end side to the whole $\mathbb{S}^{k}$ by a smooth bump function that coincides with the identity in the curve and is supported on the tubular neighbourhood \cite{lee_introduction_2012}, in order to extend the right end side to $\mathbb{S}^{k}$.
By the same tubular neighbourhood trick and bump function trick, the right side can be extended (as a $C^{\lceil{\frac{m+3}{2}}\rceil}$ function) to $\mathbb{R}^{k+1}$, and so to a $\mathbb{D}^{k+1}$ with a slightly larger radius, which is convex and compact.\\
We propose a technique to estimate it based on Taylor expansion.\\
Then, if we compute a Taylor approximation around the vector $0$, we get
\begin{equation}
    X(t)=\phi^{-1}(0)+\nabla \phi^{-1}(0)\cdot Z(t)+H[\phi^{-1}](0)\cdot (Z(t)\otimes Z(t))+\epsilon(t).
\end{equation}
We remark that the extension $g^{-1}$ is, in general, not unique, and thus the idea is to determine the coefficients of the polynomial in a smart way, being inspired by the research line of Quadratic Model order reduction \cite{geelen_operator_2023} as described in the QGPRODE section. We furthermore remark that $Z(t)$ is ergodic with ergodic measure 
\begin{equation}
\lim\limits_{T \rightarrow +\infty} \frac{1}{T} \int_{0}^{T}l(Z(t))dt:=E_{\nu}[l]=\int_{D}ld\nu \quad \forall l\in C(D)
\end{equation}
with $D\subset \mathbb{S}^{k}.$
We underline that in general, while $\tilde{B}$ is a manifold, $B$ is not.
For example, the closure of the trajectory of the above example is $\{(\sin s, \tanh s): s \geq 0\} \cup([-1,1] \times\{1\})$ which is not even a topological manifold. That's why the analysis is performed on $\tilde{B}$. The GPRODE is then applied to the closure of the embedding of $\tilde{B}$ in $\mathbb{S}^{k}$.
\section{Computation of the distribution of $\tilde{X}(t)$ in the QGPRODE}
In this section we detail how we compute $[((\operatorname{vec}(\tilde{Z}(t)^{\otimes 2} - E(t)^{\otimes 2}))^{\otimes 2})]$.

Let 
\begin{equation} 
A:=Z Z^{\top}-E E^{\top} \in \mathbb{R}^{k+1 \times k+1}, \quad Y:=\operatorname{vec}(A) \in \mathbb{R}^{(k+1)^2}
\end{equation}
We want to compute $\mathbb{E}\left[Y^{\otimes 2}\right]=\mathbb{E}\left[Y Y^{\top}\right]$ which belongs to $\mathbb{R}^{(k+1)^{2}\times (k+1)^{2}}$.
Using column-wise vectorisation

$$
Y_{(a, b)}=A_{ab}=Z_a Z_b-E_a E_b, \quad i, j=1, \ldots, d .
$$

Hence

$$
\mathbb{E}\left[Y_{(a, b)} Y_{(c, d)}\right]=\mathbb{E}\left[\left(Z_a Z_b-E_a E_b\right)\left(Z_c Z_d-E_c E_d\right)\right] .
$$
Expanding, we get
$$\mathbb{E}\left[Y_{(a, b)} Y_{(c, d)}\right]=\mathbb{E}\left[Z_a Z_b Z_c Z_d\right]-E_a E_b \mathbb{E}\left[Z_c Z_d\right]-E_c E_d \mathbb{E}\left[Z_a Z_b\right]+E_a E_b E_c E_d.$$

The term $\mathbb{E}\left[Z_a Z_b Z_c Z_d\right]$
can be computed by exploiting independence and knowing that
\begin{equation}
\mathbb{E}\left[Z_a^3\right]=\mu_a^3+3 \mu_a V_a
\end{equation}
and
\begin{equation}
\mathbb{E}\left[Z_a^4\right]=\mu_a^4+6 \mu_a^2 V_a+3 V_a^2.
\end{equation}
\\

\end{document}